\documentclass{amsart}

\usepackage{amssymb}
\usepackage[hyphens]{url}
\usepackage{amsmath}
\usepackage{amsthm}
\usepackage{mathrsfs}
\usepackage[utf8]{inputenc}
\usepackage[english]{babel}
\usepackage{cite}
\usepackage{hyperref}
\usepackage{cleveref}
\usepackage[section]{placeins}
\RequirePackage{textgreek}
\usepackage{enumitem}
\usepackage{bm}
\usepackage{nicefrac}
\usepackage{mathabx}
\usepackage[all,cmtip]{xy}
\usepackage{caption}
\usepackage{float}
\usepackage{mathtools}
\usepackage{xcolor}
\usepackage{tikz}
\usetikzlibrary{decorations.pathmorphing, intersections, arrows.meta, shapes.geometric, positioning, calc}
\tikzset{every node/.style={font=\small}, every path/.style={line width=0.7pt}, >=stealth}
\tikzset{randomcircle/.style={circle, draw, line width=1pt, minimum size=#1, rounded corners, decorate, decoration={random steps, segment length=2mm, amplitude=0.4pt}}, randomline/.style={line width=0.7pt, rounded corners, decorate,decoration={random steps, segment length=8mm, amplitude=8pt}}, mycircle/.style={circle, draw, thick, fill=blue!10, minimum size=1cm}, myrect/.style={rectangle, draw, fill=green!20, rounded corners, minimum size=1.5cm}, myarrow/.style={->, thick, >=stealth, color=red}}

\hypersetup{colorlinks=true, linkcolor=blue, citecolor=cyan, urlcolor=red}
\hypersetup{linktoc=all}
\numberwithin{equation}{subsection}
\newtheorem{thm}{Theorem}[section]
\newtheorem{lem}[thm]{Lemma}
\newtheorem{prop}[thm]{Proposition}
\newtheorem{cor}[thm]{Corollary}
\newtheorem{obs}[thm]{Observation}
\newtheorem{defn}[thm]{Definition}

\newtheorem{con}[thm]{Construction}
\newtheorem{que}[thm]{Question}
\newtheorem{rem}[thm]{Remark}

\newtheorem{mainthm}{Theorem}

\renewcommand{\themainthm}{\Alph{mainthm}}

\makeatletter
\newcommand{\ssection}[1]{%
  \par
  \vspace{0.7\baselineskip}%
  {\normalfont\scshape\centering #1\par}%
  \vspace{0.5\baselineskip}%
}
\makeatother

\newcounter{supersection}
\renewcommand{\thesupersection}{\Roman{supersection}}
\newcommand{\supersection}[1]{
  \refstepcounter{supersection}
  \phantomsection
  \addcontentsline{toc}{section}{\textbf{\thesupersection\quad #1}}
  \par\bigskip
  {\Large\bfseries \thesupersection\quad #1\par}
  \medskip
}

\author{Nachi Avraham-Re'em}
\address{Department of Mathematics, Chalmers University, G\"{o}teborg, Sweden}
\email{nachi.avraham@gmail.com \(\textstyle{or}\) nachman@chalmers.se}

\author{George Peterzil}
\address{Einstein Institute of Mathematics, The Hebrew University, Israel}
\email{george.peterzil@mail.huji.ac.il}

\thanks{Supported by ISF (grant No. 1180/22) and by the Wallenberg Foundation (KAW 2021.0258).}

\dedicatory{In memory of George Abraham (1934--1984)}

\title{The Hopf decomposition of locally compact group actions}

\subjclass[2020]{37A15, 22D40, 37A40, 37A20}

\keywords{nonsingular actions, conservative dissipative decomposition}

\begin{document}

\begin{abstract}
We develop a unified approach to the classical Hopf decomposition (also known as the conservative--dissipative decomposition) for actions of locally compact second countable groups. While the decomposition is well understood for free actions of countable groups, the extension to general actions requires new techniques and structural insights, particularly concerning recurrence and transience, cocycle behavior, and the structure of stabilizers. We establish several new characterizations and prove a structure theorem for totally dissipative actions, generalizing Krengel’s classical result for flows.
\end{abstract}

\maketitle

\tableofcontents

\section{Introduction}

At the heart of this work are two complementary properties of nonsingular actions of locally compact second countable (lcsc) groups: \textbf{conservativity} and \textbf{dissipativity}. These dynamical notions play a central role in the study of nonsingular actions, though their formulation is subtle. Over the years, many authors have proposed various definitions, conditions, and characterizations of these properties — from the classical works of Poincar\'e, Hopf, Halmos, and Maharam for a single transformation (see the bibliographical notes in \cite[\S1.3]{krengel1985ergodic} and \cite[\S1.1]{aaronson1997introduction}) to those of Schmidt \cite[\S1]{schmidt1977cocycles} and Kaimanovich \cite[\S1.1]{kaimanovich1994ergodicity, kaimanovich2010hopf} for countable groups. Krengel observed that for groups admitting lattices, the decomposition is effectively determined by the lattice subactions.\footnote{Krengel defined the Hopf decomposition for flows using \(\mathbb{Z}\)-subactions (see \cite[Lemma 2.7]{krengel1969II}).} In recent years, developments in ergodic theory, homogeneous dynamics, probability, operator algebras, and related fields have increasingly involved more general acting groups. This has highlighted the need for a systematic and detailed development of a general theory.

The idea that one can define conservativity and dissipativity for general actions of lcsc groups, and decompose every nonsingular \(G\)-space into its conservative and dissipative parts, is part of the field’s folklore. One formulation of conservativity uses a recurrence condition: from every point in a positive measure set, and outside every compact subset of \(G\), there exists a group element (not in the stabilizer of the point) that returns the point to the set (see, e.g., \cite[Definition 2.1]{danilenko2022haagerup}, \cite[\S6.3]{danilenko2022nonsingular}). Another approach asserts that the set of returning group elements has infinite volume (see, e.g., \cite[\S4.6, p.40]{quint2006overview}, \cite[\S1E]{Roblin2003}). A third classical approach is Poincar\'{e} recurrence, which replaces recurrence of points with recurrence of positive measure sets. While all of these characterizations are known to be equivalent in the countable case, the theory for lcsc groups remains underdeveloped.

In this work, we develop a unified and systematic theory of the Hopf decomposition for nonsingular actions of lcsc groups — providing a continuous analogue of the classical results of Hopf, Halmos, Maharam, and Krengel, and clarifying the relationships among existing formulations in the literature. 

We first establish Hopf’s formulation in full generality (Theorem~\ref{mthm:hopfdeco}), and in the recurrence theorems~\ref{mthm:recurrence} and \ref{mthm:recurrencepoincare} we determine its precise relation to the aforementioned recurrence conditions. In the second part of the paper, we present another perspective based on smooth Borel equivalence relations (Theorem~\ref{mthm:disssmooth}). In the next part, we introduce a characterization inspired by Maharam’s classical work, which enables a direct proof of the conservativity of the Maharam extension in full generality (Theorem~\ref{mthm:maharam}). We then give a structure theorem for totally dissipative actions, which unifies Krengel’s classical theorem for totally dissipative flows with a theorem of Glimm and Effros for ergodic actions (Theorems~\ref{mthm:krengel}, \ref{mthm:glimmeffros}, \ref{mthm:generaldiss}).

Several recent works have addressed aspects of the Hopf decomposition in lcsc groups. A framework for the Hopf decomposition is presented in the auxiliary section \cite[\S4.10]{arano2021ergodic}, though the treatment is limited both in scope and in the proofs provided. The relation of the Hopf decomposition to a particular formulation of recurrence is discussed in \cite[\S2.1]{berendschot2024phase}. A formulation in infinite measure preserving actions of unimodular groups is given in \cite[\S2.11]{blayac2025patterson} and \cite[Appendix~A]{blayac2024patterson}.

\section{Technical preliminaries}

Throughout this work, \(G\) stands for a locally compact second countable (lcsc) group with identity \(e\). Our discussion includes discrete countable groups as well as compact groups, but the majority of the theory trivializes in the latter case.

Fix a left Haar measure \(\lambda\) on \(G\), with the associated modular function \(\varDelta:G\to\mathbb{R}_{>0}\), characterized by the property
\[\int_{G}\varphi\left(g\right)\varDelta\left(g\right)d\lambda\left(g\right)=\int_{G}\varphi\small(g^{-1}\small)d\lambda\left(g\right)\text{ for all Borel function }\varphi:G\to\mathbb{R}_{\geq 0}.\]

A {\bf Borel \(G\)-space} is a standard Borel space \(X\) with a Borel map \(G\times X\to X\), \(\left(g,x\right)\mapsto g.x\) such that \(e.x=x\) and \(gh.x=g.\left(h.x\right)\) for every \(g,h\in G\) and \(x\in X\). A {\bf standard measure space} is a pair \(\left(X,\mu\right)\) consisting of a standard Borel space \(X\) and a Borel \(\sigma\)-finite measure \(\mu\) on \(X\). A a Borel set \(A\subseteq X\) is {\bf \(\mu\)-null} if \(\mu\left(A\right)=0\), {\bf \(\mu\)-positive} if it is not \(\mu\)-null, and {\bf \(\mu\)-conull} if \(X\backslash A\) is \(\mu\)-null. We write {\bf \(A=B\) modulo \(\mu\)} when \(A\triangle B\), the symmetric difference, is \(\mu\)-null. We will say that a certain property {\bf essentially} occurs when it occurs on a conull set. For proofs of the following lemma see \cite[p. 17]{krengel1985ergodic} and \cite[Lemma 1.0.7]{aaronson1997introduction}.

\begin{lem}[The exhaustion lemma]
\label{lem:exhaust}
Let \(\left(X,\mu\right)\) be a standard measure space. Every \emph{hereditary} (closed to taking subsets) collection \(\mathfrak{H}\) of Borel sets in \(X\) admits a \emph{measurable union}, namely a Borel set \(\mathcal{H}\subseteq X\) such that:
\begin{enumerate}
    \item Every element of \(\mathfrak{H}\) is a subset of \(\mathcal{H}\) modulo \(\mu\).
    \item Every \(\mu\)-positive subset of \(\mathcal{H}\) contains a \(\mu\)-positive element of \(\mathfrak{H}\).
\end{enumerate}
Moreover, \(\mathcal{H}\) is the union of countably many disjoint elements of \(\mathfrak{H}\).
\end{lem}

For a pair of measures \(\nu\) and \(\mu\) on the same standard Borel space, we denote \(\nu\ll\mu\) when \(\nu\) is absolutely continuous with respect to \(\nu\), and we denote by \(\nu\sim\mu\) equivalence of measures, that is both \(\nu\ll\mu\) and \(\mu\ll\nu\).

A {\bf nonsingular \(G\)-space} is a standard measure space \(\left(X,\mu\right)\) such that \(X\) is a Borel \(G\)-space and \(\mu\) is {\bf quasi-invariant} to the action; that is,
\[\mu\circ g\sim\mu\text{ for every }g\in G,\]
where \(\mu\circ g\) is the measure \(A\mapsto\mu\left(g.A\right)\) (\(g\) defines an invertible map). When \(\mu\circ g=\mu\) for every \(g\in G\) we will call \(\left(X,\mu\right)\) a {\bf measure preserving \(G\)-space}. Two nonsingular \(G\)-spaces \(\left(X,\mu\right)\) and \(\left(Y,\nu\right)\) are {\bf isomorphic} if there are \(G\)-invariant conull sets \(X_{o}\subseteq X\) and \(Y_{o}\subseteq Y\) and a Borel bijection \(\varphi:X_{o}\to Y_{o}\), such that \(\varphi_{\ast}\mu\sim\nu\) and for every \(g\in G\) we have \(\varphi\left(g.x\right)=g.\varphi\left(x\right)\) for \(\mu\)-a.e. \(x\in X_{o}\).

A nonsingular \(G\)-space \(\left(X,\mu\right)\) carries an associated {\bf Radon--Nikodym Cocycle},
\[\nabla:G\times X\to\mathbb{R}_{>0},\quad\nabla:\left(g,x\right)\mapsto\nabla_{g}\left(x\right),\]
which is a Borel cocycle in the sense that for every \(g,h\in G\),
\[\nabla_{e}\left(x\right)=1\text{ and }\nabla_{gh}\left(x\right)=\nabla_{g}\left(h.x\right)\cdot\nabla_{h}\left(x\right)\text{ for \emph{every} }x\in X,\]
and it coincide with the Radon--Nikodym derivatives in that for every \(g\in G\),
\[\nabla_{g}\left(x\right)=\frac{d\mu\circ g}{d\mu}\left(x\right)\text{ for }\mu\text{-a.e. }x\in X.\]
The defining property of \(\nabla\) can be put in the convenient form
\begin{equation}
\label{eq:nabla}\tag{\(\dagger\)}
\begin{aligned}
&\int\nolimits_{X}\nabla_{g}\left(x\right)f_{0}\left(g.x\right)f_{1}\left(x\right)d\mu\left(x\right)=\int\nolimits_{X}f_{0}\left(x\right)f_{1}\left(g^{-1}.x\right)d\mu\left(x\right),
\end{aligned}
\end{equation}
\[\text{ for all Borel functions } f_{0},f_{1}:X\to\mathbb{R}_{\geq 0}\text{ and }g\in G.\]

The existence of a pointwise-defined cocycle \(\nabla\) realizing the Radon–Nikodym derivatives follows from the Mackey Cocycle Theorem \cite[Lemma 5.26, p. 179]{varadarajan1968geometry}.

A Borel set \(A\subseteq X\) is {\bf \(G\)-invariant} if \(G.A= A\), and {\bf \(\mu\)-almost \(G\)-invariant} if \(\mu\left(g.A\triangle A\right)=0\) for each \(g\in G\). The following important facts are well-known.\footnote{The first part is \cite[Prop.~8.3]{einsiedler2011ergodic}, together with Varadarajan’s compact model theorem \cite[Ch.~V, \S3, Thm.~5.7]{varadarajan1968geometry}. The second part is \cite[Lemma~B.8]{zimmer2013ergodic}, whose proof uses a theorem of Kallman \cite[Thm.~A.5]{zimmer2013ergodic}, now known as the Arsenin--Kunugui Uniformization Theorem \cite[Thm.~(18.18)]{kechris2012descriptive}.}

\begin{prop}
\label{prop:invariance}
Let \(\left(X,\mu\right)\) be a nonsingular \(G\)-space.
\begin{enumerate}
    \item For every \(\mu\)-almost \(G\)-invariant set \(A\subseteq X\), there is a \(G\)-invariant set \(B\subseteq X\), such that \(\mu\left(A\triangle B\right)=0\).
    \item For every Borel set \(A\subseteq X\), there is a Borel set \(B\subseteq A\) with \(\mu\left(A\backslash B\right)=0\) and \(G.B\) is a (\(G\)-invariant) Borel set.
\end{enumerate}
\end{prop}

\supersection{Two general forms of the Hopf decomposition}

\section{Hopf's formulation}

Let \(\left(X,\mu\right)\) be a nonsingular \(G\)-space. Denote by \(L_{+}^{1}\left(X,\mu\right)\) the class of (everywhere defined, strictly) positive functions in \(L^{1}\left(X,\mu\right)\), and for \(f\in L_{+}^{1}\left(X,\mu\right)\) define
\[S_{f}^{G}\left(x\right)\coloneqq\int_{G}\frac{d\mu\circ g}{d\mu}\left(x\right)f\left(g.x\right)d\lambda\left(g\right),\quad x\in X.\]
Then every such \(f\) yields a decomposition of \(X\) into
\[\mathcal{C}_{f}\coloneqq \left\{ x\in X:S_{f}^{G}\left(x\right)=+\infty\right\}\text{ and }\mathcal{D}_{f}\coloneqq X\backslash\mathcal{C}_{f}.\]
These are Borel sets (see \cite[Theorem (17.25)]{kechris2012descriptive}). The following theorem, originally proved by Hopf for a single transformation, is part of the folklore of ergodic theory (cf. \cite[Theorem 4.30]{arano2021ergodic}).

\begin{mainthm}[Hopf decomposition]
\label{mthm:hopfdeco}
For every nonsingular \(G\)-space \(\left(X,\mu\right)\), there exists an essentially unique \(G\)-invariant decomposition
\[X=\mathcal{C}\sqcup\mathcal{D},\]
depending only on the measure class of \(\mu\),\footnote{Meaning that the decompositions of equivalent measures is the same modulo \(\mu\); note still that the average transform \(S_{f}^{G}\) defining the decomposition does depend on the measure itself.} with the property that
\[\mu\left(\mathcal{C}\triangle\mathcal{C}_{f}\right)=0\text{ and }\mu\left(\mathcal{D}\triangle\mathcal{D}_{f}\right)=0\text{ for every }f\in L_{+}^{1}\left(X,\mu\right).\]
In particular, \(\mathcal{C}_{f}\) and \(\mathcal{D}_{f}\) are independent of \(f\in L_{+}^{1}\left(X,\mu\right)\) modulo \(\mu\).
\end{mainthm}

Following the standard terminology (see the introduction of \cite{danilenko2023ergodic}), a nonsingular \(G\)-space \(\left(X,\mu\right)\) with the Hopf decomposition \(X=\mathcal{C}\sqcup\mathcal{D}\) is called:
\begin{description}
    \item[\(\bullet\) conservative] if \(\mu\left(\mathcal{D}\right)=0\);
    \item[\(\bullet\) dissipative] if \(\mu\left(\mathcal{D}\right)>0\); and
    \item[\(\bullet\) totally dissipative \textnormal{(or} completely dissipative\textnormal{)}] if \(\mu\left(\mathcal{C}\right)=0\).
\end{description}

\section{The recurrence theorem}

We now move to the second form of the Hopf decomposition, as a \emph{recurrence} property. Let us first establish useful notation: For a Borel set \(A\subseteq X\) denote
\[R_{A}\left(x\right)\coloneqq \left\{g\in G:x\in g.A\right\},\quad x\in X.\]
A set \(A\) should be considered as recurrent (in a pointwise manner) when \(R_{A}\left(x\right)\) is large in \(G\) for every \(x\in A\). As we mentioned previously, there are two natural ways to formulate this property: A Borel set \(A\subseteq X\) is:
\begin{description}
    \item[\(\bullet\) recurrent] if \(R_{A}\left(x\right)\) is not relatively compact for \(\mu\)-a.e. \(x\in A\).
    \item[\(\bullet\) Haar-recurrent] if \(\lambda\left(R_{A}\left(x\right)\right)=+\infty\) for \(\mu\)-a.e. \(x\in A\).
    \item[\(\bullet\) transient] if \(R_{A}\left(x\right)\) is relatively compact for \(\mu\)-a.e. \(x\in X\).
    \item[\(\bullet\) Haar-transient] if \(\lambda\left(R_{A}\left(x\right)\right)<+\infty\) for \(\mu\)-a.e. \(x\in X\).\footnote{In \cite[\S4.6 p.40]{quint2006overview} and \cite[\S1E]{Roblin2003} this is called a {\bf wandering set}. Since this does not fully align with the notion of a wandering set for countable groups, we prefer to use a different term.}
\end{description}

Note that transience is defined over all of \(X\), while recurrence is defined inside the reference set \(A\). Note also that a subset of a (Haar-)transient set is (Haar-)transient, thus the class of (Haar-)transient sets is \emph{hereditary}.

\begin{mainthm}[The recurrence theorem]
\label{mthm:recurrence}
Let \(\left(X,\mu\right)\) be a nonsingular \(G\)-space. The Hopf decomposition \(X=\mathcal{C}\sqcup\mathcal{D}\) (as in Theorem~\ref{mthm:hopfdeco}) is determined by:
\begin{enumerate}
    \item Every subset of \(\mathcal{C}\) is recurrent (equivalently, Haar-recurrent).
    \item Every \(\mu\)-positive subset of \(\mathcal{D}\) contains a \(\mu\)-positive transient (equivalently, Haar-transient) subset.
\end{enumerate}
\end{mainthm}

\begingroup
\footnotesize
\noindent \textbf{Note:} Theorem~\ref{mthm:recurrence} should be interpreted to mean that any of the four combinations of Haar-recurrent/recurrent and Haar-transient/transient characterizes the Hopf decomposition.
\endgroup\\

The Hopf decomposition can be described also using Poincar\'{e} Recurrence, which treats recurrence setwise. Also here, there are two natural formulations. A Borel set \(A\subseteq X\) is:
\begin{description}
    \item[\(\bullet\) Poincar\'{e} recurrent] if outside every compact set in \(G\) there is some element \(g\) such that \(\mu\left(A \cap g.A\right)>0\).
    \item[\(\bullet\) Poincar\'{e} Haar-recurrent] if \(\lambda\left(g\in G:\mu\left(A\cap g.A\right)>0\right)=+\infty\).
\end{description}

In the same way that the behaviour of the system does not distinguish between recurrence and Haar-recurrence, it likewise does not distinguish between Poincar\'{e} recurrence and Poincar\'{e} Haar-recurrence. While this does not follow directly from Theorem~\ref{mthm:recurrence}, the ideas in its proof will inform the proof of the following theorem:

\renewcommand{\themainthm}{B'}
\begin{mainthm}
\label{mthm:recurrencepoincare}
A nonsingular \(G\)-space \(\left(X,\mu\right)\) is conservative iff every \(\mu\)-positive Borel subset of \(X\) is Poincar\'{e} recurrent (equivalently, Poincar\'{e} Haar-recurrent).
\end{mainthm}
\renewcommand{\themainthm}{\Alph{mainthm}}

\ssection{Proofs}

We will prove Theorems~\ref{mthm:hopfdeco} and \ref{mthm:recurrence} simultaneously. The proof of the recurrence theorem~\ref{mthm:recurrence} given here is only to its version with Haar-recurrent and Haar-transient sets. In order to deduce that the same holds with recurrent and transient sets, as well as the Poincar\'{e} recurrence theorem~\ref{mthm:recurrencepoincare}, we will need the third general form of the Hopf decomposition, and thus the final proof of Theorem~\ref{mthm:recurrence} will be given in the end of Section~\ref{sct:finalrecurrence}.

\begin{thm}
\label{thm:recurrence}
Let \(\left(X,\mu\right)\) be a nonsingular \(G\)-space. For an arbitrary \(f\in L_{+}^{1}\left(X,\mu\right)\) and an arbitrary Borel set \(A\subseteq X\), the following are equivalent:
\begin{enumerate}
    \item \(A\subseteq\mathcal{C}_{f}\) modulo \(\mu\).
    \item For every \(\mu\)-positive set \(B\subseteq A\) it holds that
    \[\int_{G}1_{B}\left(g^{-1}.x\right)d\lambda\left(g\right)=+\infty\text{ for }\mu\text{-positively many }x\in X.\]
    \item For every \(\mu\)-positive set \(B\subseteq A\) it holds that
    \[\int_{G}1_{B}\left(g^{-1}.x\right)d\lambda\left(g\right)=+\infty\text{ for }\mu\text{-a.e. }x\in B.\]
\end{enumerate}
\end{thm}

Using the notation we introduced before, for a Borel set \(A\subseteq X\) and \(x\in X\) write
\[R_{A}\left(x\right)\coloneqq \left\{ g\in G:x\in g.A\right\},\text{ so that }\lambda\left(R_{A}\left(x\right)\right)=\int_{G}1_{A}\left(g^{-1}.x\right)d\lambda\left(g\right).\]
Observe the basic identity
\begin{equation}
\label{eq:riden}
R_{A}\left(g.x\right)=gR_{A}\left(x\right),\quad x\in X, g\in G.
\end{equation}

\begin{defn}
Let \(\mathfrak{T}\) be the hereditary collection of Haar-transient sets in \(\left(X,\mu\right)\). By the exhaustion lemma~\ref{lem:exhaust} it admits a measurable union that we denote by \(\mathcal{T}\).
\end{defn}

\begin{prop}
\label{prop:recurrence}
For every \(f\in L_{+}^{1}\left(X,\mu\right)\) it holds that \(\mu\left(\mathcal{T}\triangle\mathcal{D}_{f}\right)=0\).
\end{prop}

\begin{proof}[Proof of Proposition~\ref{prop:recurrence}]
We start by showing that \(\mu\left(\mathcal{D}_{f}\backslash\mathcal{T}\right)=0\). Let \(A\subseteq\mathcal{T}^{\complement}\) be an arbitrary Borel set and put
\[A_{0}\coloneqq \left\{x\in A:S_{f}^{G}\left(x\right)<+\infty\right\}.\]
Suppose toward a contradiction that \(A_{0}\) is \(\mu\)-positive. Passing to some \(\mu\)-positive subset \(B\subseteq A_{0}\) with \(\mu(B)<+\infty\), using the formula \eqref{eq:nabla} we obtain
\begin{align*}
&+\infty>\int_{B}S_{f}^{G}\left(x\right)d\mu\left(x\right)\\
&\qquad\quad=\iint\nolimits_{G\times X}\nabla_{g}\left(x\right)f\left(g.x\right)1_{B}\left(x\right)d\lambda\otimes\mu\left(g,x\right)\\
&\qquad\quad=\iint\nolimits_{G\times X}f\left(x\right)1_{B}\left(g^{-1}.x\right)d\lambda\otimes\mu\left(g,x\right)\\
&\qquad=\int_{X}f\left(x\right)\lambda\left(R_{B}\left(x\right)\right)d\mu\left(x\right).
\end{align*}
Since \(f\) is positive on a \(\mu\)-conull set, it follows that \(\lambda\left(R_{B}\left(x\right)\right)<+\infty\) for \(\mu\)-a.e. \(x\in X\). Thus, \(A\supseteq B\in\mathfrak{T}\) which is a contradiction to that \(A\subseteq\mathcal{T}^{\complement}\). We deduce that \(S_{f}^{G}\left(x\right)=+\infty\) for \(\mu\)-a.e. \(x\in A\), namely \(A\subseteq\mathcal{C}_{f}\). Since \(A\subseteq\mathcal{T}^{\complement}\) is an arbitrary \(\mu\)-positive set it readily follows that \(\mu\left(\mathcal{T}^{\complement}\backslash\mathcal{C}_{f}\right)=0\), hence \(\mu\left(\mathcal{D}_{f}\backslash\mathcal{T}\right)=0\).

We now show that \(\mu\left(\mathcal{T}\backslash\mathcal{D}_{f}\right)=0\). Fix an arbitrary \(T\in\mathfrak{T}\). For \(r>0\) let
\[X_{r}\coloneqq \left\{x\in X:\lambda\left(R_{T}\left(x\right)\right)\leq r\right\}\text{ and }T_{r}\coloneqq T\cap X_{r}.\]
These are Borel sets (see \cite[Theorem (17.25)]{kechris2012descriptive}). Since \(T\in\mathfrak{T}\) it follows that \(X_{r}\nearrow X\) as \(r\nearrow +\infty\) modulo \(\mu\), hence \(T_{r}\nearrow T\) as \(r\nearrow +\infty\) modulo \(\mu\).

For a fixed \(r>0\) and every \(x\in X\), using \eqref{eq:riden} and the invariance of \(\lambda\) we have
\begin{align*}
g\in R_{T_{r}}\left(x\right)
&\iff\left(g^{-1}.x\in T\right)\wedge\left(\lambda\left(R_{T}\left(g.x\right)\right)\leq r\right)\\
&\iff\left(g^{-1}.x\in T\right)\wedge\left(\lambda\left(R_{T}\left(x\right)\right)\leq r\right),
\end{align*}
hence
\[R_{T_{r}}\left(x\right)=\begin{cases}
R_{T}\left(x\right) & \lambda\left(R_{T}\left(x\right)\right)\leq r\\
\emptyset & \text{otherwise}
\end{cases},\]
so that \(\lambda\left(R_{T_{r}}\left(x\right)\right)\leq r\) for every \(x\in X\). Using the formula \eqref{eq:nabla}, for every \(r>0\),
\begin{equation}
\label{eq:maharam}
\begin{aligned}
&\int_{T_{r}}S_{f}^{G}\left(x\right)d\mu\left(x\right)\\
&\qquad=\iint\nolimits_{G\times X}\nabla_{g}\left(x\right)f\left(g.x\right)1_{T_{r}}\left(x\right)d\lambda\otimes\mu\left(g,x\right)\\
&\qquad=\iint\nolimits_{G\times X}f\left(x\right)1_{T_{r}}\left(g^{-1}.x\right)d\lambda\otimes\mu\left(g,x\right)\\
&\qquad=\int_{X}f\left(x\right)\lambda\left(R_{T_{r}}\left(x\right)\right)d\mu\left(x\right)\\
&\leq r\left\Vert f\right\Vert_{L^{1}\left(X,\mu\right)}<+\infty.
\end{aligned}
\end{equation}
It follows that \(S_{f}^{G}\left(x\right)<+\infty\) for \(\mu\)-a.e. \(x\in T_{r}\). Since \(r\) is arbitrary it follows that \(S_{f}^{G}\left(x\right)<+\infty\) for \(\mu\)-a.e. \(x\in T\). By the definition of \(\mathcal{D}_{f}\) this implies that \(\mu\left(T\backslash\mathcal{D}_{f}\right)=0\). Since \(T\in\mathfrak{T}\) is arbitrary we deduce that \(\mu\left(\mathcal{T}\backslash\mathcal{D}_{f}\right)=0\).
\end{proof}

\begin{proof}[Proof of Theorem~\ref{thm:recurrence}]
Let \(A\subseteq X\) be a \(\mu\)-positive set. By Proposition~\ref{prop:recurrence} we have \(\mu\left(A\cap\mathcal{D}_{f}\right)=0\iff\mu\left(A\cap\mathcal{T}\right)=0\). Thus, \(\mu\left(A\cap\mathcal{D}_{f}\right)=0\) iff \(A\) contains no \(\mu\)-positive Haar-transient set. This is precisely condition (2) for \(A\) as in the theorem. Thus, we have established the equivalence of conditions (1) and (2) for \(A\).

It is obvious that the failure of condition (2) for \(A\) implies the failure of condition (3) for \(A\). Let us show the converse. Thus, suppose that \(A\) is a \(\mu\)-positive set containing no Haar-transient set, and let \(B\subseteq A\) be an arbitrary \(\mu\)-positive set. Consider the set
\[T_{B}\coloneqq \left\{x\in B:\lambda\left(R_{B}\left(x\right)\right)<+\infty\right\}.\]
As the notation suggests, we claim that \(T_{B}\) is a Haar-transient set. Indeed, let \(x\in X\) be arbitrary and, as we argued in the proof of Proposition~\ref{prop:recurrence} with the identity \eqref{eq:riden}, one directly verifies that
\[R_{T_{B}}\left(x\right)=\begin{cases}
R_{B}\left(x\right) & \lambda\left(R_{B}\left(x\right)\right)<+\infty\\
\emptyset & \text{otherwise}
\end{cases},\]
so that \(\lambda\left(R_{T_{B}}\left(x\right)\right)<+\infty\) for every \(x\in T_{B}\), concluding that \(T_{B}\) is Haar-transient. As \(T_{B}\subseteq B\subseteq A\) and \(A\) contains no \(\mu\)-positive Haar-transient set, \(\lambda\left(T_{B}\right)=0\). This is precisely condition (3) for \(A\).
\end{proof}

\begin{proof}[Proof of Theorem~\ref{mthm:hopfdeco}]
Recall the identity \eqref{eq:riden} from which it follows that if \(T\) is a Haar-transient set then, for every \(g\in G\), also \(g.T\) is a Haar-transient set. This readily implies that \(\mathcal{T}\) is almost \(\mu\)-invariant. By Proposition~\ref{prop:invariance}(1) there is a \(G\)-invariant set \(\mathcal{T}'\subseteq X\) with \(\mu\left(\mathcal{T}\triangle\mathcal{T}'\right)=0\), thus \(\mathcal{T}'\) is also a measurable union of the class of Haar-transients sets. We then put
\[\mathcal{D}=\mathcal{T}'\text{ and }\mathcal{C}=X\backslash \mathcal{D}.\]
It follows from Proposition~\ref{prop:recurrence} directly that this \(G\)-invariant decomposition constitutes the Hopf decomposition of \(\left(X,\mu\right)\). The fact that this decomposition depends only on the measure class of \(\mu\) follows directly from the recurrence theorem~\ref{thm:recurrence}.
\end{proof}

\supersection{Third general form of the Hopf decomposition}

\section{Wandering sets and Borel transversals}

In the following we turn into a descriptive set-theoretic point of view. When \(G\) is countable and \(\left(X,\mu\right)\) is essentially free, the dissipative part can be presented as
\[\mathcal{D}=\bigcup\nolimits_{g\in G}g.W_{o}\text{ modulo }\mu,\]
where \(W_{o}\) is a {\bf wandering set}, which is a strong form of transient set: \(\left\{g.W_{o}:g\in G\right\}\) are pairwise disjoint (see \cite[\S1.3]{krengel1985ergodic}, \cite[\S1.1 \& \S1.6]{aaronson1997introduction}; see also \cite{shelah1982measurable, Weiss}).\footnote{In Aaronson's approach \cite[\S1.6]{aaronson1997introduction}, totally dissipative nonsingular \(G\)-spaces are measurable unions of wandering sets, hence they are essentially free (see \cite[Proposition 1.6.1]{aaronson1997introduction}). Our approach uses transient sets instead, hence allows non-essentially free nonsingular \(G\)-spaces to be totally dissipative. The example one should bear in mind is the measure preserving \(G\)-space \(\left(G/F,\#\right)\), where \(F\) is a finite subgroup of \(G\) and \(\#\) is the counting measure. It admits no wandering sets unless \(F\) is trivial, but every singleton \(gF\in G/F\) forms a transient set (cf. \cite[Remarks 19--22]{kaimanovich2010hopf}).}

When it comes to uncountable groups, of course there cannot be any \(\mu\)-positive wandering set. Nevertheless, for \(G=\mathbb{R}\) Krengel made essential use of wandering sets of zero measure which are {\bf Borel transversals}, namely a Borel set that meets the orbit of every point exactly once. He then showed in \cite[\S3]{krengel1968I}, \cite{krengel1969II}, that if \(\left(X,\mu\right)\) is a totally dissipative nonsingular \(\mathbb{R}\)-space, then it admits a Borel transversal \(W_{o}\) on which one can construct a measure \(\nu_{o}\) in such a way that \(\left(X,\mu\right)\) would be isomorphic as nonsingular \(G\)-spaces to the nonsingular \(G\)-space
\begin{equation}
\label{eq:transflow}
\left(W_{o}\times \mathbb{R},\nu_{o}\otimes\mathrm{Lebesgue}\right)\text{ with the action }t.\left(w,s\right)=\left(w,t+s\right).
\end{equation}

Recall that the {\bf orbit equivalence relation} of a Borel \(G\)-space \(X\) is
\[E_{G}^{X}\coloneqq \left\{\left(x,g.x\right)\in X\times X:x\in X,g\in G\right\}.\]
We say that \(E_{G}^{X}\) is {\bf smooth} if it admits a Borel transversal, i.e. a Borel set \(W_{o}\subseteq X\) that meets every orbit exactly once. This property has a few useful characterizations that will be mentioned in Theorem~\ref{thm:smooth} below. Nonsingular \(G\)-spaces whose orbit equivalence relation is smooth were long been studied, since the celebrated works of Glimm and Effros \cite{glimm1961locally, effros1965transformation} (see also \cite[\S5]{feldman1978orbit} and \cite[\S2.1]{zimmer2013ergodic}).

\section{Formulation of the third form}

Smoothness of the orbit equivalence relation will be shown in the upcoming theorem to be a necessary condition for dissipativity, but it is not sufficient: compactness of the stabilizers is also required to fully capture dissipativity. In some sources this was taken as the definition of dissipativity (see \cite[Definition 2.1(i)]{danilenko2022haagerup}), and in others it was observed that it is sufficient for dissipativity (see \cite[\S2.11]{blayac2025patterson}), which is the easy part of the following theorem.

\renewcommand{\themainthm}{C}
\begin{mainthm}
\label{mthm:disssmooth}
For every nonsingular \(G\)-space \(\left(X,\mu\right)\), the following are equivalent:
\begin{enumerate}
    \item \(\left(X,\mu\right)\) is totally dissipative.
    \item Modulo \(\mu\), \(E_{G}^{X}\) is smooth and all stabilizers are compact.
\end{enumerate}
\end{mainthm}
\renewcommand{\themainthm}{\Alph{mainthm}}

Since \(\mathbb{Z}^{d}\) and \(\mathbb{R}^{d}\) have no nontrivial compact subgroups, we obtain:

\begin{cor}[Krengel--Rosinski]
When \(A=\mathbb{Z}^{d}\) or \(\mathbb{R}^{d}\), a nonsingular \(A\)-space \(\left(X,\mu\right)\) is totally dissipative iff it is essentially free and \(E_{A}^{X}\) is essentially smooth.
\end{cor}

Another corollary of Theorem~\ref{mthm:disssmooth} arises from Becker's Glimm–Effros Dichotomy in orbit equivalence relations \cite[Theorem 6.5.2]{gao2008invariant}. Let \(E_{0}\) denote the tail equivalence relation on \(\{0,1\}^{\mathbb{N}}\). We use \(\sqsubset_{\mathrm{B}}\) to denote Borel reducibility between equivalence relations (see \cite[\S6]{gao2008invariant} for background and details).

\begin{cor}
Let \(\left(X,\mu\right)\) be a nonsingular \(G\)-space such that essentially all stabilizers are compact (e.g. essentially free actions). Then \(\left(X,\mu\right)\) is conservative iff \(E_{0}\sqsubset_{\mathrm{B}}E_{G}^{C}\) for every \(\mu\)-positive \(G\)-invariant set \(C\subseteq X\).
\end{cor}

\section{Smooth orbit equivalence relations}

For every Borel \(G\)-space, the associated orbit equivalence relation \(E_{G}^{X}\) is Borel, i.e. it forms a Borel set in \(X\times X\), and each of its equivalence classes, i.e. each orbit \(G.x\coloneqq \left\{g.x:g\in G\right\}\), is Borel in \(X\) (see \cite[Exercise 3.4.6, Theorem 3.3.2]{gao2008invariant}). The following well-known theorem includes a few equivalent formulations of the notion of a \emph{smooth} Borel equivalence relation. While it can be formulated more generally, we focus solely on orbit equivalence relations induced from lcsc group actions.

\begin{thm}
\label{thm:smooth}
For every Borel \(G\)-space \(X\), the following are equivalent:
\begin{enumerate}
    \item \(E_{G}^{X}\) is \emph{smooth}: there is a Borel map \(s\) from \(X\) to some standard Borel space, such that
    \[\left(x,x'\right)\in E\iff s\left(x\right)=s\small(x'\small)\text{ for all }x,x'\in X.\] 
    \item \(E_{G}^{X}\) is \emph{countably separated}: there exists a countable collection of Borel functions \(u_{1},u_{2},\dotsc\) from \(X\) to some standard Borel space, such that
    \[\left(x,x'\right)\in E\iff \left(\forall n\in\mathbb{N},\,u_{n}\left(x\right)=u_{n}\small(x'\small)\right)\text{ for all }x,x'\in X.\]
    \item \(E_{G}^{X}\) admits a \emph{Borel selector}: a Borel function \(s:X\to X\) such that
    \[\left(x,s\left(x\right)\right)\in E\text{ and }\left(x,x'\right)\in E\iff s\left(x\right)=s\small(x'\small)\text{ for all }x,x'\in X.\]
    \item \(E_{G}^{X}\) admits a \emph{Borel transversal}: a Borel subset of \(X\) that intersects every orbit in exactly one point.
    \item The space of orbits \(X/E_{G}^{X}\) is standard Borel.\footnote{The measurable sets in \(X/E_{G}^{X}\) are collections of orbits whose union is Borel in \(X\).}
\end{enumerate}
\end{thm}

We refer to the proofs of the various parts of Theorem~\ref{thm:smooth}. The equivalence \(\left(1\right)\iff\left(2\right)\) is \cite[Proposition 5.4.4]{gao2008invariant}. The equivalence \(\left(1\right)\iff\left(3\right)\) is the most difficult part of this theorem and a version of this in the presence of a measure on \(X\) is attributed to von Neumann (see \cite[Appendix A]{zimmer2013ergodic} and in particular Theorem A.9). Without a measure this is a result of Burgess, which holds even for Polish groups (see \cite[Corollary 5.4.12]{gao2008invariant}); for a proof due to Kechris see \cite[Proposition 5.4.10 and Theorem 5.4.11]{gao2008invariant}. The equivalence \(\left(3\right)\iff\left(4\right)\) appears in the proof of \cite[Proposition 3.4.5]{gao2008invariant}. For the implication \(\left(4\right)\implies\left(5\right)\), note that whenever
\(W_{o}\) is a transversal for \(E_{G}^{X}\), the natural map \(W_{o}\to X/E_{G}^{X}\) that takes a point to its orbit is a bijection, and thus when \(W_{o}\) is a Borel set this becomes a Borel bijection inducing a standard Borel structure on \(X/E_{G}^{X}\). Finally, the implication \(\left(5\right)\implies\left(1\right)\) follows since the natural projection \(s:X\to X/E_{G}^{X}\) makes it smooth.

\bigskip

Let \(\mathbf{F}\left(G\right)\) denote the \(G\)-space of closed subsets of \(G\), where \(G\) acts by left translation. This space carries a standard Borel structure (the \emph{Effros space}; see \cite[Exercise (12.7) \& \S12.E]{kechris2012descriptive}), relative to which the following sets are Borel:
\begin{itemize}
    \item The set \(\mathbf{F}_{\mathrm{gr}}\left(G\right)\) of all closed subgroups of \(G\);
    \item The set \(\mathbf{K}\left(G\right)\) of all compact subsets of \(G\);
    \item The set \(\mathbf{K}_{\mathrm{gr}}\left(G\right)\) of all compact subgroups of \(G\).
\end{itemize}

In addition, the following theorem is satisfied by this structure of \(\mathbf{F}\left(G\right)\), and will be essential to what follows (see \cite[Theorems 3.3.2 \& 8.2.1]{gao2008invariant}). Recall that for a Borel \(G\)-space \(X\), the stabilizer of a point \(x\in X\) is the subgroup
\[G_{x}\coloneqq\left\{g\in G:g.x=x\right\}.\]

\begin{thm}[Miller, Becker--Kechris]
\label{thm:miller}
Let \(X\) be a Borel \(G\)-space. Then \(G_{x}\) is a closed subgroup for all \(x\in X\), and the map \(X\to\mathbf{F}_{\mathrm{gr}}\left(G\right)\), \(x\mapsto G_{x}\), is Borel.
\end{thm}

\begin{con}
\label{cnst:psispace}
For any Borel map \(\psi_{o}:W_{o}\to\mathbf{F}_{\mathrm{gr}}\left(G\right)\) on some standard Borel space \(W_{o}\), define a Borel \(G\)-space
\[X_{o}\coloneqq\left\{\left(w,g\psi_{o}\left(w\right)\right):w\in W_{o},g\in G\right\}\subseteq X\times F_{\mathrm{gr}}(G),\]
with the action given by
\[g.\left(w,h\psi_{o}\left(w\right)\right)=\left(w,gh\psi_{o}\left(w\right)\right),\quad g\in G,\,\left(w,h\psi_{o}\left(w\right)\right)\in X_{o}.\]
\end{con}

\begin{lem}
\label{lem:standard}
\(X_{o}\) is a (standard) Borel \(G\)-space.
\end{lem}

\begin{proof}
Consider the Borel map
\[\psi_{o}:W_{o}\times G\to W_{o}\times\mathbf{F}\left(G\right), \quad\psi_{o}\left(w,g\right)=\left(w,g\psi_{o}\left(w\right)\right).\]
Thus, \(X_{o}=\psi_{o}\left(W_{o}\times G\right)\). In order to deduce that \(X_{o}\) is a Borel subset of \(W_{o}\times\mathbf{F}\left(G\right)\), note that the fiber of every point \(\left(W_{o},g_{0}\psi_{o}\left(W_{o}\right)\right)\in W_{o}\times\mathbf{F}\left(G\right)\) under \(\psi_{o}\) is
\begin{align*}
\psi_{o}^{-1}\left(W_{o},g_{0}\psi_{o}\left(W_{o}\right)\right)
&=\left\{ \left(w,g\right):\left(w,g\psi_{o}\left(w\right)\right)=\left(W_{o},g_{0}\psi_{o}\left(W_{o}\right)\right)\right\}\\
&=\left\{ W_{o}\right\} \times\left\{ g\in G:g\psi_{o}\left(W_{o}\right)=g_{0}\psi_{o}\left(W_{o}\right)\right\}\\
&=\left\{ W_{o}\right\} \times g_{0}\psi_{o}\left(W_{o}\right).
\end{align*}
Pick any \(\sigma\)-compact Polish topology on \(W_{o}\) that induces its Borel structure (which always exists by the isomorphism theorem of standard Borel spaces), and consider the product topology of this with the given (\(\sigma\)-compact) topology of \(G\). Then with this product topology the space \(W_{o}\times G\) becomes \(\sigma\)-compact and Polish. It is also clear that \(\left\{ W_{o}\right\} \times g_{0}\psi_{o}\left(W_{o}\right)\) is closed \(W_{o}\times G\), hence \(\sigma\)-compact. We thus found that all fibers of \(\Psi\) are \(\sigma\)-compact in some Polish topology that is compatible with the Borel structure of \(W_{o}\times G\). It then follows from the Arsenin--Kunugui Theorem (see \cite[Theorem (18.18)]{kechris2012descriptive}, \cite[Theorem 7.5.1]{gao2008invariant}, \cite[Corollary A.6]{zimmer2013ergodic}) that the image of \(W_{o}\times G\) under \(\psi_{o}\), namely \(X_{o}\), is a Borel set in \(W_{o}\times\mathbf{F}\left(G\right)\).
\end{proof}

\begin{obs}
\label{obs:stab}
The stabilizer of a point \(\left(w,h\psi_{o}\left(w\right)\right)\in X_{o}\) is
\[G_{\left(w,h\psi_{o}\left(w\right)\right)}=\left\{ g\in G:gh\psi_{o}\left(w\right)=h\psi_{o}\left(w\right)\right\}=\psi_{o}\left(w\right)^{h},\]
which is also the stabilizer of \(h\psi_{o}\left(w\right)\) in the coset \(G\)-space \( G/\psi_{o}\left(w\right)\).
\end{obs}

\begin{thm}
\label{thm:smoothactions}
For every Borel \(G\)-space \(X\), the orbit equivalence relation \(E_{G}^{X}\) is smooth iff \(X\) is isomorphic to a Borel \(G\)-space as in Construction~\ref{cnst:psispace}.
\end{thm}

The following lemma is a consequence of the Ryll--Nardzewski Selection Theorem.

\begin{lem}
\label{lem:grpselect}
Let \(X\) be a Borel \(G\)-space such that \(E_{G}^{X}\) is smooth with a Borel transversal \(W_{o}\). There exists a Borel map \(r:X\to G\) such that \(r\left(x\right).x\in W_{o}\) for every \(x\in X\), and the map
\[s:X\to W_{o},\quad s\left(x\right)\coloneqq r\left(x\right).x,\]
is a Borel selector. Such a map \(r:X\rightarrow G\) necessarily satisfies:
\begin{enumerate}
    \item \(r\left(g.x\right)g.x=r\left(x\right).x\) for every \(g\in G\) and \(x\in X\).
    \item \(r\left(g.x\right)gr\left(x\right)^{-1}\in G_{r\left(x\right).x}\) for every \(g\in G\) and \(x\in X\).
\end{enumerate} 
\end{lem}

\begin{proof}
For \(x,x'\in X\) denote \(G_{x,x'}=\{g\in G:g.x=x'\}\). Then either \(G_{x,x'}=\emptyset\) or that for an arbitrary \(g\in G_{x,x'}\) we have \(G_{x,x'}=gG_{x}\), in particular it is closed in \(G\) by Theorem~\ref{thm:miller}, so we view \(G_{x,x'}\) as an element of \(\mathbf{F}\left(G\right)\). By the Kuratowski--Ryll-Nardzewski Selection Theorem (see \cite[Theorem 1.4.6]{gao2008invariant}) there is a Borel map 
\[\xi:\mathbf{F}\left(G\right)\backslash\{\emptyset\}\to G\text{ satisfying }\xi\left(C\right)\in C\text{ for every }C\in \mathbf{F}\left(G\right)\backslash\{\emptyset\}.\]
Fix a Borel selector \(s:X\to X\) for \(E_{G}^{X}\), and for \(x\in X\), since \(G_{x,s\left(x\right)}\neq\emptyset\), put
\[r:X\to G,\quad r\left(x\right)\coloneqq\xi\left(G_{x,s\left(x\right)}\right).\]
Then \(r\left(x\right).x=s\left(x\right)\in W_{o}\) for all \(x\in X\). It is Borel since it is composed of
\[r:X\xrightarrow{x\mapsto\left(x,s\left(x\right)\right)}X\times X\xrightarrow{\small(x,x'\small)\mapsto G_{x,x'}}\mathbf{F}\left(G\right)\xrightarrow{C\mapsto\xi\left(C\right)}G,\]
where the second is Borel by Becker--Kechris Theorem (see \cite[Theorem 8.2.1]{gao2008invariant}). Finally, the properties of \(r\) follow directly from the defining property of \(W_{o}\).
\end{proof}

\begin{proof}[Proof of Theorem~\ref{thm:smoothactions}]
First, if \(X\) is (isomorphic to) the Borel \(G\)-space \(X_{o}\) that is constructed out of \(\psi_{o}:W_{o}\to\mathbf{F}_{\mathrm{gr}}\left(G\right)\), then the graph
\[\left\{ \left(w,\psi_{o}\left(w\right)\right):w\in W_{o}\right\} \subseteq X\]
of \(\psi_{o}\) is a transversal for \(\left(X,\mu\right)\), so \(E_{G}^{X}\) is  smooth. For the converse, suppose \(X\) is a Borel \(G\)-space and \(E_{G}^{X}\) is smooth, pick a Borel transversal \(W_{o}\subseteq X\) and define
\[\psi_{o}:W_{o}\to\mathbf{F}_{\mathrm{gr}}\left(G\right),\quad\psi_{o}\left(w\right)=G_{w},\quad w\in W_{o}.\]
By Theorem~\ref{thm:miller} this is a Borel map, so let \(X_{o}\) be the Borel \(G\)-space associated with \(\psi_{o}:W_{o}\to\mathbf{F}_{\mathrm{gr}}\left(G\right)\). Pick a Borel function \(r:X\to G\) as in Lemma~\ref{lem:grpselect}, and define
\[\iota:X\to X_{o}\subseteq W_{o}\times\mathbf{F}\left(G\right),\quad\iota\left(x\right)=\big(r\left(x\right).x,r\left(x\right)^{-1}G_{r\left(x\right).x}\big).\]
In proving that \(\iota\) is an isomorphism we will constantly use the properties of \(r\) in Lemma~\ref{lem:grpselect}. First, \(\iota\) is equivariant since for every \(g\in G\) and \(x\in X\) we have
\[\iota\left(g.x\right)=\big(r\left(g.x\right)g.x,r\left(g.x\right)^{-1}G_{r\left(g.x\right)g.x}\big)=\big(r\left(x\right).x,gr\left(x\right)^{-1}G_{r\left(x\right).x}\big)=g.\iota\left(x\right).\]To see that \(\iota\) is surjective, note that for every \(g\in G\) and \(w\in W_{o}\),    \[w=r\left(w\right).w=r\left(g.w\right).w=r\left(g.w\right)g.w,\]
and in particular \(r\left(g.w\right)g\in G_{w}\), hence
\begin{align*}
\iota\left(g.w\right)
&=\big(r\left(g.w\right)g.w,r\left(g.w\right)^{-1}G_{r\left(g.w\right)g.w}\big)\\
&=\big(w,r\left(g.w\right)^{-1}G_{w}\big)=\left(w,gG_{w}\right)=\left(w,g\psi_{o}\left(w\right)\right).
\end{align*}
To see that \(\iota\) is injective, define \(\jmath:X_{o}\to X\) by
\[\quad\jmath\left(w,g\psi_{o}\left(w\right)\right)=g.w,\quad\left(w,g\psi_{o}\left(w\right)\right)\in X_{o}.\]
We verify that \(\jmath\) is well-defined. If \(\left(w,g\psi_{o}\left(w\right)\right)=\left(w',g'\psi_{o}\left(w'\right)\right)\) then \(g\psi_{o}\left(w\right)=g'\psi_{o}\left(w'\right)=g'\psi_{o}\left(w\right)\), hence \(g.w=g'.w\) so that
\[\jmath\left(w,g\psi_{o}\left(w\right)\right)=g.w=g'.w=g'.w'=\jmath\left(w',g'\psi_{o}\left(w'\right)\right).\]
We finally see that \(\jmath\) is the inverse of \(\iota\), since for every \(x\in X\),
\[\jmath\left(\iota\left(x\right)\right)=\jmath\big(r\left(x\right).x,r\left(x\right)^{-1}G_{r\left(x\right).x}\big)=r\left(x\right)^{-1}r\left(x\right).x=x.\qedhere\]
\end{proof}

\ssection{Proofs}

The following theorem was proved by Krengel for \(G=\mathbb{R}\) \cite[Satz 3.1]{krengel1968I}, \cite[Satz 4.2]{krengel1969II}, and by Rosinski for \(G=\mathbb{R}^{d}\) \cite[Theorem 2.2]{rosinski2000decomposition}. In what follows we provide a proof for a general \(G\) which is greatly inspired by Rosinski's proof.

\begin{thm}[Following Krengel--Rosinski]
\label{thm:disssmooth}
If \(\left(X,\mu\right)\) is a totally dissipative nonsingular \(G\)-space, then \(E_{G}^{X}\) is essentially smooth.
\end{thm}

\begin{proof}
In the proof we will work with the conull set \(\mathcal{T}\) as in Proposition~\ref{prop:recurrence}, and show that the orbit equivalence relation restricted to \(\mathcal{T}\) is countably separated by Borel functions \(\left\{u_{n,m}:n,m\in\mathbb{N}\right\}\), from which smoothness follows by Theorem~\ref{thm:smooth}.

\subsubsection*{Step 1:} (constructing \(\left\{u_{n,m}:n,m\in\mathbb{N}\right\}\)).

We recall that \(\mathcal{T}\) is the measurable union of the class of Haar-transient sets of \(\left(X,\mu\right)\). As such, the exhaustion lemma~\ref{lem:exhaust} guarantees that \(\mathcal{T}\) is the union of countably many Haar-transient sets, \(\left\{T_{1},T_{2},\dotsc\right\}\). Let us also apply Varadarajan's compact model theorem \cite[Theorem 5.7, p. 160]{varadarajan1968geometry}. Thus, there is a compact separable \(G\)-space \(\widebar{X}\) (i.e. there is a jointly continuous action of \(G\) on \(\widebar{X}\)), and a Borel \(G\)-embedding of \(X\) in \(\widebar{X}\). We thus can fix a countable base \(\mathcal{U}=\left\{U_{1},U_{2},\dotsc\right\}\) for the topology of \(\widebar{X}\), and we may further assume that it is closed to finite unions and finite intersections. We now put, for every \(n,m\in\mathbb{N}\), the function
\[u_{n,m}:X\to\mathbb{R},\quad u_{n,m}\left(x\right)\coloneqq \lambda\left(R_{U_{n}\cap T_{m}}\left(x\right)\right)=\int_{G}1_{U_{n}\cap T_{m}}\left(g^{-1}.x\right)d\lambda\left(g\right).\]
We recall that each \(U_{n}\cap T_{m}\), being a subset of a Haar-transient set, is Haar-transient, so that \(u_{n,m}\) takes only finite values on the \(\mu\)-conull set \(\mathcal{T}\).

\subsubsection*{Step 2:} (a positivity property of \(\left\{u_{n,m}:n,m\in\mathbb{N}\right\}\)).

We will show that for every \(x_{o}\in\mathcal{T}\), there are \(n_{0},m_{0}\in\mathbb{N}\) with \(u_{n_{0},m_{0}}\left(x_{o}\right)>0\). Since \(\left\{U_{n}\cap T_{m}:n,m\in\mathbb{N}\right\}\) covers \(\mathcal{T}\), for every \(g\in G\) also \(\left\{g.U_{n}\cap g.T_{m}:n,m\in\mathbb{N}\right\}\) covers \(\mathcal{T}\). Thus, for every \(x_{o}\in\mathcal{T}\) we have
\[\sum\nolimits_{n,m\in\mathbb{N}}1_{U_{n}\cap T_{m}}\left(g^{-1}.x_{o}\right)\geq1.\]
It then follows that
\[\sum\nolimits_{n,m\in\mathbb{N}}u_{n,m}\left(x_{o}\right)=\int_{G}\left(\sum\nolimits_{n,m\in\mathbb{N}}1_{U_{n}\cap T_{m}}\left(g^{-1}.x_{o}\right)\right)d\lambda\left(g\right)=+\infty,\]
which readily implies the desired property.

\subsubsection*{Step 3:} (each \(u_{n,m}\) is constant along orbits).

Let \(\left(x_{o},h.x_{o}\right)\in E_{G}^{X}\) for some \(x_{o}\in X\) and \(h\in G\). Using \eqref{eq:riden}, for every \(n,m\in\mathbb{N}\),
\begin{align*}
u_{n,m}\left(h.x_{o}\right)
&=\lambda\left(R_{U_{n}\cap T_{m}}\left(h.x_{o}\right)\right)\\
&=\lambda\left(h.R_{U_{n}\cap T_{m}}\left(x_{o}\right)\right)=\lambda\left(R_{U_{n}\cap T_{m}}\left(x_{o}\right)\right)=u_{n,m}\left(x_{o}\right).
\end{align*}

\subsubsection*{Step 4:} (\(E_{G}^{X}\) is separated by \(\left\{u_{n,m}:n,m\in\mathbb{N}\right\}\)).

Let \(\left(x_{0},x_{1}\right)\notin E_{G}^{X}\). Using Step 2, we pick \(n_{0},m_{0}\in\mathbb{N}\) for \(x_{0}\) with \(u_{n_{0},m_{0}}\left(x_{0}\right)>0\). If \(u_{n_{0},m_{0}}\left(x_{0}\right)\neq u_{n_{0},m_{0}}\left(x_{1}\right)\) we are done, so assume otherwise that 
\[c\coloneqq u_{n_{0},m_{0}}\left(x_{0}\right)=u_{n_{0},m_{0}}\left(x_{1}\right)>0.\]
By the inner regularity of \(\lambda\), there is a compact symmetric set \(K\subset G\) such that
\[\int_{K}1_{U_{n_{0}}\cap T_{m_{0}}}\left(g^{-1}.x_{i}\right)>c/2,\quad i\in\left\{ 0,1\right\}.\]
By the continuity of the action on \(\overline{X}\), the sets \(K.x_{0}\) and \(K.x_{1}\) are compact in \(\overline{X}\) and, since \(\left(x_{0},x_{1}\right)\notin E_{G}^{X}\), necessarily \(K.x_{0}\cap K.x_{1}=\emptyset\).

As \(\mathcal{U}\) is closed to finite unions, we may pick \(n_{1}\in\mathbb{N}\) such that
\[K.x_{0}\subseteq U_{n_{1}}\text{ and }K.x_{1}\cap U_{n_{1}}=\emptyset.\]
As \(\mathcal{U}\) is closed to finite intersections, we may pick \(n_{2}\in\mathbb{N}\) such that
\[U_{n_{0}}\cap U_{n_{1}}=U_{n_{2}}.\]
We then claim that \(u_{n_{2},m_{0}}\) assigns different values to \(x_{0}\) and \(x_{1}\).

On one hand, from \(K.x_{0}\cap U_{n_{0}}\subseteq U_{n_{0}}\cap U_{n_{1}}=U_{n_{2}}\) we have
\begin{align*}
u_{n_{2},m_{0}}\left(x_{0}\right)
&=\int_{G}1_{U_{n_{2}}\cap T_{m_{0}}}\left(g^{-1}.x_{0}\right)d\lambda\left(g\right)\\
&\geq\int_{G}1_{K.x_{0}\cap U_{n_{0}}\cap T_{m_{0}}}\left(g^{-1}.x_{0}\right)d\lambda\left(g\right)\\
&\geq\int_{K}1_{U_{n_{0}}\cap T_{m_{0}}}\left(g^{-1}.x_{0}\right)d\lambda\left(g\right)>c/2.
\end{align*}
On the other hand, from \(K.x_{1}\cap U_{n_{2}}\subseteq K.x_{1}\cap U_{n_{1}}=\emptyset\) we have
\[\int_{K}1_{U_{n_{2}}\cap T_{m_{0}}}\left(g^{-1}.x_{1}\right)d\lambda\left(g\right)\leq\int_{G}1_{K.x_{1}\cap U_{n_{2}}\cap T_{m_{0}}}\left(g^{-1}.x_{1}\right)d\lambda\left(g\right)=0,\]
hence
\begin{align*}
u_{n_{2},m_{0}}\left(x_{1}\right)
&=\int_{G\backslash K}1_{U_{n_{2}}\cap T_{m_{0}}}\left(g^{-1}.x_{1}\right)d\lambda\left(g\right)\\
&\leq\int_{G\backslash K}1_{U_{n_{0}}\cap T_{m_{0}}}\left(g^{-1}.x_{1}\right)d\lambda\left(g\right)\\
&=u_{n_{0},m_{0}}\left(x_{1}\right)-\int_{K}1_{U_{n_{0}}\cap T_{m_{0}}}\left(g^{-1}.x_{1}\right)d\lambda\left(g\right)<c-c/2=c/2.
\end{align*}
Thus, \(u_{n_{2},m_{0}}\left(x_{0}\right)\neq u_{n_{2},m_{0}}\left(x_{1}\right)\), concluding the proof of Step 4.

Finally, by the virtue of Steps 3 and 4, \(E_{G}^{X}\) is countably separated by Borel functions, and from Theorem~\ref{thm:smooth} we deduce that it is smooth.
\end{proof}

\begin{thm}
\label{thm:compstab}
If \(\left(X,\mu\right)\) is a totally dissipative nonsingular \(G\)-space, then the (\(G\)-invariant) set \(X_{\mathrm{c}}\subseteq X\) of points whose stabilizers are compact is \(\mu\)-conull.
\end{thm}

For the proof of Theorem~\ref{thm:compstab} we will need the following elementary fact, generalizing that an lcsc group with finite Haar measure is compact. The proof is an adaptation of \cite[Chapter II, \S5, Proposition 4]{nachbin1965haar}.

\begin{lem}
\label{lem:compcriter}
Every lcsc group admitting a Haar-integrable continuous homomorphism into \(\mathbb{R}_{>0}\) is compact.
\end{lem}

\begin{proof}
Let \(H\) be an lcsc group with a left Haar measure \(\lambda_{H}\), and \(\psi:H\to\mathbb{R}_{>0}\) a continuous homomorphism with \(\int_{H}\psi d\lambda_{H}<+\infty\). If \(\psi\) is bounded it is the trivial homomorphism hence \(\lambda_{H}\left(H\right)=\int_{H}\psi d\lambda_{H}<+\infty\),
in which case \(H\) is compact. Suppose \(\psi\) is unbounded and, since it is continuous, \(H\) is non-compact and \(\psi\) is unbounded outside every compact set in \(H\). Fix a compact set \(K\subset H\) with
\[I_{K}\coloneqq \int_{K}\psi d\lambda_{H}\in\left(0,+\infty\right),\]
and define a sequence \(h_{1},h_{2},\dotsc\in H\) such that \(h_{1}K,h_{2}K,\dotsc\) are pairwise disjoint and \(\psi\left(h_{i}\right)\geq 1\) for \(i=1,2,\dotsc\). For every \(n\in\mathbb{N}\) we have \(\sum\nolimits_{i=1}^{n}1_{K}\left(h_{i}^{-1}h\right)\leq1\) for all \(h\in H\), and using the left invariance of \(\lambda_{H}\) and that \(\psi\) is a homomorphism we get
\begin{align*}
+\infty
&>\int_{H}\psi\left(h\right)d\lambda_{H}\left(h\right)\\
&\geq\sum\nolimits_{i=1}^{n}\int_{H}1_{K}\left(h_{i}^{-1}h\right)\psi\left(h\right)d\lambda_{H}\left(h\right)\\
&=\int_{K}\psi\left(h\right)d\lambda_{H}\left(h\right)\cdot\sum\nolimits_{i=1}^{n}\psi\left(h_{i}\right)\geq I_{K}\cdot n.
\end{align*}
Since \(I_{K}\in\left(0,+\infty\right)\) and \(n\in\mathbb{N}\) is arbitrary we got to a contradiction.
\end{proof}

In the following proof we make use of the construction of \emph{measures on homogeneous spaces} and \emph{Weil's formula} (see Theorem~\ref{thm:meashom} in the \hyperref[Appendix: Measures on Homogeneous Spaces]{Appendix}).

\begin{proof}[Proof of Theorem~\ref{thm:compstab}]
Pick some \(f\in L_{+}^{1}\left(X,\mu\right)\) and let \(x\in\mathcal{D}_{f}\) be fixed from now until (nearly) the end of the proof. In light of Proposition~\ref{prop:dissacim} we assume that \(\mu\) is \(G\)-invariant and thus \(\nabla_{g}\equiv 1\) for every \(g\in G\). We apply the construction of measures on homogeneous spaces as in Theorem~\ref{thm:meashom} to the group pair \(G_{x}\lessdot G\) (recall Theorem~\ref{thm:miller}). Thus, \(G\) has the Haar measure \(\lambda\) and the corresponding modular function \(\varDelta\), and for \(G_{x}\) we fix a Haar measure \(\lambda_{x}\) and the corresponding modular function \(\varDelta_{x}\). Pick a rho-function \(\rho_{x}:G\to\mathbb{R}_{>0}\) for \(G_{x}\lessdot G\) and a quasi-invariant measure \(\kappa_{x}\) on \(G/G_{x}\) with respect to \(\rho_{x}\). Applying Weil's formula \eqref{eq:weil1} to the function \(G\to\mathbb{R}_{>0}\), \(g\mapsto f\left(g.x\right)\), we obtain
\begin{align*}
S_{f}^{G}\left(x\right)
&=\int_{G}f\left(g.x\right)d\lambda\left(g\right)\\
&=\int_{G/G_{x}}\Big[\int_{G_{x}}f\left(gc.x\right)\varDelta\left(c\right)\varDelta_{x}\left(c\right)^{-1}d\lambda_{x}\left(c\right)\Big]\rho\left(g\right)^{-1}d\kappa_{x}\left(gG_{x}\right)\\
&=\Big[\int_{G/G_{x}}f\left(g.x\right)\rho\left(g\right)^{-1}d\kappa_{x}\left(gG_{x}\right)\Big]\cdot\Big[\int_{G_{x}}\varDelta\left(c\right)\varDelta_{x}\left(c\right)^{-1}d\lambda_{x}\left(c\right)\Big].
\end{align*}
Since \(x\in\mathcal{D}_{f}\) we have \(S_{f}^{G}\left(x\right)<+\infty\), and since all terms are strictly positive,
\[\int_{G_{x}}\varDelta\left(c\right)\varDelta_{x}\left(c\right)^{-1}d\lambda_{x}\left(c\right)<+\infty.\]
However, \(G_{x}\to\mathbb{R}_{>0}\), \(c\mapsto\varDelta\left(c\right)\varDelta_{x}\left(c\right)^{-1}\), is a continuous homomorphism, then \(G_{x}\) is compact by Lemma~\ref{lem:compcriter}. Finally, letting \(X_{\mathrm{c}}\) be the set of points in \(X\) whose stabilizer is compact (which is Borel by Theorem~\ref{thm:miller}), we have shown that \(\mathcal{D}_{f}\subseteq X_{\mathrm{c}}\) modulo \(\mu\). Since \(\left(X,\mu\right)\) is totally dissipative we deduce that \(X_{\mathrm{c}}\) is \(\mu\)-conull.
\end{proof}

\begin{proof}[Proof of Theorem~\ref{mthm:disssmooth}]
The implication (1)\(\implies\)(2) is the virtue of Theorems~\ref{thm:disssmooth} and~\ref{thm:compstab}. Let us show the implication (2)\(\implies\)(1). By the assumption, the \(G\)-invariant set \(X_{\mathrm{c}}\) of points whose stabilizers are compact is \(\mu\)-conull, and by the other assumption we may find a Borel transversal \(W_{o}\) on some \(\mu\)-conull subset of \(X_{\mathrm{c}}\), thus the stabilizer of \emph{every} point in \(W_{o}\) is compact. Recall that by the identity \eqref{eq:riden}, a translation of a Haar-transient set is again Haar-transient. Then by the property of \(W_{o}\) as a transversal, in order to prove that \(\mu\)-a.e. \(x\in X\) is contained in a \(\mu\)-positive Haar-transient set it is sufficient to show that this is true for every point in \(W_{o}\). Since \(G.W_{o}=X\) and \(G\) is lcsc, it is sufficient to show that \(C.W_{o}\) is Haar-transient for every compact set \(C\subset G\). Indeed, note the identity
\[R_{C.W_{o}}\left(w\right)=G_{w}C^{-1},\quad w\in W_{o},\]
from which it follows that for \(\mu\)-a.e. \(x\in X\), picking \(g_{0}\in G\) with \(g_{0}.x\in W_{o}\), we have
\[R_{C.W_{o}}\left(x\right)=g_{0}^{-1}R_{C.W_{o}}\left(g_{0}.x\right)=g_0^{-1}G_{g_{0}.x}C^{-1}.\]
Since \(G_{g_{0}.x}\) and \(C^{-1}\) are both compact so is \(g_{0}^{-1}G_{g_{0}.x}C^{-1}\), and we conclude that \(C.W_{o}\) is transient and a fortiori Haar-transient. This shows that \(\left(X,\mu\right)\) is totally dissipative, completing the proof of Theorem~\ref{mthm:disssmooth}.
\end{proof}

\subsection{Final proofs of the recurrence theorems}
\label{sct:finalrecurrence}

Theorem~\ref{mthm:disssmooth} allows us to complete the proof of Theorems~\ref{mthm:recurrence} and \ref{mthm:recurrencepoincare}.

\begin{proof}[Proof of Theorem~\ref{mthm:recurrence}]
As we mentioned before, the proof of the version with Haar-recurrent and Haar-transient sets follows directly from the recurrence theorem~\ref{thm:recurrence} and from \Cref{prop:recurrence} by taking \(f\) to be an indicator function for the appropriate subset, so we prove now the version with recurrent and transient sets. Let \(\mathfrak{T}\) and \(\mathfrak{S}\) be the collections of Haar-transient sets and of transient sets, respectively. Since both are hereditary they admit measurable unions that we denote \(\mathcal{T}\) and \(\mathcal{S}\), respectively. As follows from \Cref{mthm:hopfdeco} and \Cref{prop:recurrence}, the Hopf decomposition \(X=\mathcal{C}\sqcup\mathcal{D}\) is determined by the property that \(\mu\left(\mathcal{D}\triangle\mathcal{T}\right)=0\), so it is sufficient to show that \(\mathcal{T}=\mathcal{S}\) modulo \(\mu\). On one hand, as every transient set is Haar-transient, it is clear that \(\mathcal{S}\subseteq\mathcal{D}\) modulo \(\mu\). Conversely, in the proof of the recurrence theorem~\ref{thm:recurrence} we have shown that \(\mu\)-a.e. point in \(\mathcal{D}\) is contained in a \(\mu\)-positive transient set, so \(\mathcal{D}\) is a measurable union of the transient sets, thus \(\mathcal{D}=\mathcal{S}\) modulo \(\mu\).

From this we immediately see that every \(\mu\)-positive subset of \(\mathcal{D}\) has a \(\mu\)-positive transient subset. We need to show that every subset of \(\mathcal{C}\) is recurrent. Given a Borel set \(A\subseteq \mathcal{C}\), we consider the set
\[T_{A}\coloneqq \left\{x\in A: R_{T}\left(x\right)\text{ is relatively compact}\right\},\]
and for \(r>0\) consider the set
\[T_{A,r}\coloneqq \left\{ x\in T_{A}:R_{T_{A}}\left(x\right)\subseteq B_{r}\right\},\]
where \(B_{r}\) is the ball of radius \(r>0\) with respect to some compatible proper metric on \(G\).\footnote{Which generally exists by Struble's theorem \cite{struble1974metrics}.} We then see that
\[g\in R_{T_{A,r}}\left(x\right)\iff\left(g^{-1}.x\in T_{A}\right)\wedge\left(R_{T_{A}}\left(x\right)\subseteq B_{r}\right),\]
hence
\[R_{T_{A,r}}\left(x\right)=\begin{cases}
R_{T_{A}}\left(x\right) & R_{T_{A}}\left(x\right)\subseteq B_{r}\\
\emptyset & \text{otherwise},
\end{cases},\]
and in any case \(R_{T_{A,r}}\left(x\right)\subseteq B_{r}\), thus \(T_{A,r}\) is transient and is contained in \(A\), and by the assumption \(\mu\left(T_{A,r}\right)=0\). Since \(T_{A,r}\nearrow T_{A}\) as \(r\nearrow +\infty\) we deduce that \(\mu\left(T_{A}\right)=0\). Thus, \(A\) is recurrent.
\end{proof}

\begin{proof}[Proof of Theorem~\ref{mthm:recurrencepoincare}]
We will relate to a nonsingular \(G\)-space \(\left(X,\mu\right)\) as \emph{Poincar\'{e} recurrent} (\emph{Poincar\'{e} Haar-recurrent}) if every \(\mu\)-positive Borel subset of \(X\) is Poincar\'{e} recurrent (\emph{Poincar\'{e} Haar-recurrent}, respectively). We will then show that
\begin{align*}
&\text{Poincar\'{e} Haar-recurrence }\implies\text{ Poincar\'{e} recurrence }\\
&\qquad\quad\implies\text{ conservativity }\implies\text{ Poincar\'{e} Haar-recurrence.}
\end{align*}

The first implication is clear. As for the second implication, if \(\left(X,\mu\right)\) is not conservative then as we have shown in the proof of Theorem \ref{mthm:disssmooth} there can be found a \(\mu\)-positive transient set \(T\subseteq X\) and \(r>0\) such that \(R_{T}\left(x\right)\subseteq B_{r}\) for \emph{every} \(x\in T\). This means that \(T\cap g.T=\emptyset\) for every \(g\in G\backslash B_{r}\), hence \(T\) is not Poincar\'{e} recurrent. As for the third implication, first note that by the Fubini Theorem, for every Borel set \(A\subseteq X\) we have the identity
\begin{equation}
\label{eq:poincare}
\lambda\left(g\in G:\mu\left(A\cap g.A\right)>0\right)=\int_{X}1_{A}\left(x\right)\lambda\left(R_{A}\left(x\right)\right)\mu\left(x\right).
\end{equation}
It follows that, if \(\left(X,\mu\right)\) is not Poincar\'{e} Haar-recurrent, there can be found a \(\mu\)-positive Borel set \(A\subseteq X\) for which the quantity in \eqref{eq:poincare} is finite, thus \(\lambda\left(R_{A}\left(x\right)\right)<+\infty\) for \(\mu\)-a.e. \(x\in A\), meaning that \(A\) is not Haar-recurrent. By Theorem~\ref{mthm:recurrence} it follows that \(\left(X,\mu\right)\) is not conservative.
\end{proof}

\supersection{Fourth general form of the Hopf decomposition}

\section{Maharam extensions}

The Maharam extension is a classical construction in ergodic theory due to D. Maharam \cite{maharam1964incompressible}, and it associates with every nonsingular \(G\)-space a canonical (infinite) measure preserving \(G\)-space as follows. For a nonsingular \(G\)-space \(\left(X,\mu\right)\) with Radon--Nikodym cocycle \(\nabla\), define its {\bf Maharam extension} to be the measure preserving \(G\)-space
\[\big(\widetilde{X},\widetilde{\mu}\big)\coloneqq\left(X\times\mathbb{R},\mu\otimes\eta\right),\text{ where }d\eta\left(t\right)\coloneqq e^{t}dt,\]
and with the action given by
\[g.\left(x,t\right)=\left(g.x,t-\log\nabla_{g}\left(x\right)\right).\]
One can routinely check that this indeed forms a measure preserving \(G\)-space.

The first key property of this construction, which was famously proved by Maharam~\cite[Theorem 2]{maharam1964incompressible} for a single transformation, and by Schmidt \cite[Theorem 4.2 (p. 47), Theorem 5.5 (p. 56)]{schmidt1977cocycles} for countable groups, is that the Maharam extension of a conservative nonsingular \(G\)-space is conservative. In other words, the Radon--Nikodym cocycle of a conservative nonsingular \(G\)-space is \emph{recurrent} \cite[\S4]{schmidt1977cocycles}.

\begin{rem}
\label{rem:skewprod}
Given any additive cocycle \(\Psi:G\times X\to\mathbb{R}\), \(\Psi:\left(g,x\right)\mapsto\Psi_{g}\left(x\right)\), one can construct the skew-product Borel \(G\)-space \(X\ltimes_{\Psi}\mathbb{R}\) with action defined by
\[g.\left(x,t\right)=\left(g.x,t+\Psi_{g}\left(x\right)\right).\]
If \(\left(X,\mu\right)\) is a nonsingular \(G\)-space with Hopf decomposition \(X=\mathcal{C}\sqcup\mathcal{D}\), and \(\eta\) is an absolutely continuous measure on \(\mathbb{R}\), then the Hopf decomposition of the nonsingular \(G\)-space \(\left(X\ltimes_{\Psi}\mathbb{R},\mu\otimes\eta\right)\) takes the form
\[X\ltimes_{\Psi}\mathbb{R}=\left(C_{\Psi}\times\mathbb{R}\right)\sqcup\left(D_{\Psi}\times\mathbb{R}\right),\text{ for }G\text{-invariant sets }C_{\Psi}\subseteq\mathcal{C}\text{ and }D_{\Psi}\supseteq\mathcal{D}.\footnote{This can be shown using the fact that the action commutes with the \(\mathbb{R}\)-flow \(s.\left(x,t\right)=\left(x,t+s\right)\).}\]
It is generally possible that \(C_{\Psi}=\emptyset\) and \(\mathcal{C}=X\). Nevertheless, Maharam's theorem~\ref{mthm:maharam} asserts that this is never the case when \(\Psi\) is the Radon--Nikodym cocycle.
\end{rem}

\section{Maharam's theorem}

In establishing the conservativity of the Maharam extension, we will make essential use of a property observed by Maharam~\cite[\S8]{maharam1964incompressible}, according to which if \(S\) is a nonsingular transformation of \(\left(X,\mu\right)\), then for \(\mu\)-almost every \(x\in X\), not only does
\[\sum\nolimits_{n\in\mathbb{N}}\textstyle{\frac{d\mu \circ S^{n}}{d\mu}(x)}=+\infty,\]
but moreover, there exists \(r_{x}>0 \) such that
\[\#\big\{n\in\mathbb{N}:{\textstyle \frac{d\mu\circ S^{n}}{d\mu}\left(x\right)}\geq r_{x}\big\}=+\infty.\]
This was generalized by Kaimanovich~\cite[Theorem 29]{kaimanovich2010hopf} to actions of countable groups on probability spaces. We now formulate Maharam's theorem for general lcsc groups via a significant generalization of this property to arbitrary \(\mathbb{R}_{>0}\)-valued cocycles. Recall that a (\(\mathbb{R}_{>0}\)-valued, multiplicative) Borel cocycle of a Borel \(G\)-space \(X\), is a Borel map \(\Psi:G\times X\to\mathbb{R}_{>0}\), \(\Psi:\left(g,x\right)\mapsto\Psi_{g}\left(x\right)\), such that
\[\Psi_{e}\left(x\right)=1\text{ and }\Psi_{gh}\left(x\right)=\Psi_{g}\left(h.x\right)\Psi_{h}\left(x\right)\text{ for all }x\in X\text{ and }g,h\in G.\]

The following fourth general form of the Hopf decomposition was proved by Maharam for a single transformation (cf. \cite[Proposition 4.34]{arano2021ergodic}).

\renewcommand{\themainthm}{D}
\begin{mainthm}
\label{mthm:maharam}
For every nonsingular \(G\)-space \(\left(X,\mu\right)\), the following are equivalent:
\begin{enumerate}
    \item \(\left(X,\mu\right)\) is conservative.
    \item Every Borel cocycle \(\Psi:G\times X\to\mathbb{R}_{>0}\) satisfies the following property:\\
    for \(\mu\)-a.e. \(x\in X\) there exists \(r=r_{\left(\Psi,x\right)}>0\) such that
    \[\lambda\left(g\in G:\Psi_{g^{-1}}\left(x\right)\geq r\right)=+\infty.\]
    \item For every \(f\in L_{+}^{1}\left(X,\mu\right)\) the following  property holds:\\
    for \(\mu\)-a.e. \(x\in X\) there exists \(r=r_{\left(f,x\right)}>0\) such that
    \[\lambda\big(g\in G:\nabla_{g}\left(x\right)f\left(g.x\right)\geq r\big)=+\infty.\]
    \item The Maharam extension \(\big(\widetilde{X},\widetilde{\mu}\big)\) of \(\left(X,\mu\right)\) is conservative.
\end{enumerate}
\end{mainthm}
\renewcommand{\themainthm}{\Alph{mainthm}}

\begin{proof}[Proof of \(\left(1\right)\implies\left(2\right)\) in Theorem~\ref{mthm:maharam}]
Fix a cocycle \(\Psi\), and for \(r>0\) and \(x\in X\) denote \(Q_{r}\left(x\right)\coloneqq\left\{ g\in G:\Psi_{g^{-1}}\left(x\right)\geq r\right\}\). Let \(E\) be the set where (2) fails for \(\Psi\),
\[E=E_{\Psi}\coloneqq\left\{ x\in X:\forall r>0,\,\lambda\left(Q_{r}\left(x\right)\right)<+\infty\right\}.\]
This is a Borel set by \cite[Theorem (17.25)]{kechris2012descriptive}. For the rest of the proof we will show that \(E\subseteq\mathcal{D}\) modulo \(\mu\), where \(\mathcal{D}\) is the dissipative part of \(\left(X,\mu\right)\).

Observe that \(E\) is \(G\)-invariant. Indeed, by the cocycle property, for arbitrary \(x\in X\) and \(g_{o}\in G\), for every \(r>0\) we have
\[Q_{r}\left(g_{o}.x\right)=\left\{ g\in G:\Psi_{g^{-1}}\left(g_{o}.x\right)\geq r\right\} =\left\{ g\in G:\Psi_{g^{-1}g_{o}}\left(x\right)\geq \Psi_{g_{o}}\left(x\right)r\right\},\]
and therefore, substituting \(g\mapsto g_{o}g\) and using the left-invariance of \(\lambda\),
\[\lambda\left(Q_{r}\left(g_{o}.x\right)\right)=\lambda\big(g\in G:\Psi_{g^{-1}}\left(x\right)\geq r\Psi_{g_{o}}\left(x\right)\big)=\lambda\left(Q_{\Psi_{g_{o}}\left(x\right)r}\left(x\right)\right).\]
This readily implies that \(E\) is \(G\)-invariant. Fix an arbitrary \(r>0\), and define
\[\tau_{r}:E\to\mathbb{R}_{>0},\quad\tau_{r}\left(x\right)\coloneqq\inf\left\{s>0:\lambda\left(Q_{s}\left(x\right)\right)\leq r\right\}.\]
This is a Borel function by \cite[Theorem (17.25)]{kechris2012descriptive}. Since \(\lambda\left(Q_{s}\left(x\right)\right)\nearrow+\infty\) as \(s\searrow 0\) for every \(x\in E\), by the assumption \(\tau_{r}\left(x\right)\) is strictly positive. Define
\[Q_{r}^{\prime}\left(x\right)\coloneqq Q_{\tau_{r}\left(x\right)}\left(x\right)\text{ for all }x\in E,\]
and define the set
\[T_{r}\coloneqq\bigcup\nolimits_{x\in E}Q_{r}^{\prime}\left(x\right)^{-1}.x,\]
which is analytic hence \(\mu\)-measurable (see \cite[Theorem (21.10)]{kechris2012descriptive}). Note that since \(E\) is \(G\)-invariant, \(T_{r}\subseteq E\). Let us show that \(T_{r}\) is Haar-transient. First note that by the assumption and using that \(\tau_{r}\left(x\right)>0\),
\begin{equation}
\label{eq:dprop}
\lambda\big(Q_{r}^{\prime}\left(x\right)\big)<+\infty\text{ for every }x\in E.
\end{equation}
Additionally, we claim that
\begin{equation}
\label{eq:tprop}
R_{T_{r}}\left(y\right)\subseteq\begin{cases}
Q_{r}^{\prime}\left(y\right) & y\in E\\
\emptyset & y\notin E
\end{cases}\quad\text{for every }y\in X.
\end{equation}
Let us justify \eqref{eq:tprop}. Start by noting that by definition, for every \(x\in E\),
\begin{align*}
Q_{r}^{\prime}\left(x\right)
&=\left\{ g\in G:\Psi_{g^{-1}}\left(x\right)\geq\tau_{r}\left(x\right)\right\}\\
&=\left\{ g\in G:\lambda\left(h\in G:\Psi_{h^{-1}}\left(x\right)\geq\Psi_{g^{-1}}\left(x\right)\right)\leq r\right\},
\end{align*}
and hence, for every \(g_{o}\in G\) and \(x\in E\),
\begin{align*}
g_{o}g\in Q_{r}^{\prime}\left(g_{o}.x\right)
&\iff\lambda\big(h\in G:\Psi_{h^{-1}}\left(g_{o}.x\right)\geq\Psi_{g^{-1}g_{o}^{-1}}\left(g_{o}.x\right)\big)\leq r\\
\small(\mathrm{i}\small)&\iff\lambda\big(h\in G:\Psi_{h^{-1}}\left(g_{o}.x\right)\geq\Psi_{g^{-1}}\left(x\right)\Psi_{g_{o}^{-1}}\left(g_{o}.x\right)\big)\leq r\\
\small(\mathrm{ii}\small)&\iff\lambda\big(h\in G:\Psi_{h^{-1}g_{o}}\left(x\right)\geq\Psi_{g^{-1}}\left(x\right)\big)\leq r\\
\small(\mathrm{iii}\small)&\iff\lambda\big(h\in G:\Psi_{h^{-1}}\left(x\right)\geq\Psi_{g^{-1}}\left(x\right)\big)\leq r\\
&\iff g\in Q_{r}^{\prime}\left(x\right),
\end{align*}
where in \(\small(\mathrm{i}\small)\) and \(\small(\mathrm{ii}\small)\) we used the cocycle property of \(\Psi\), and in \(\small(\mathrm{iii}\small)\) we used the substitution \(h\mapsto g_{o}h\) and the left-invariance of \(\lambda\). Thus, we have found that
\[Q_{r}^{\prime}\left(g_{o}.x\right)=g_{o}Q_{r}^{\prime}\left(x\right)\text{ for every }g_{o}\in G\text{ and }x\in E.\]
Now \eqref{eq:tprop} follows: let \(y\in X\) be arbitrary and suppose \(g\in R_{T_{r}}\left(y\right)\), thus \(g^{-1}.y\in T_{r}\), so there are \(x\in E\) and \(g_{o}\in Q_{r}^{\prime}\left(x\right)^{-1}\) such that \(g^{-1}.y=g_{o}.x\). Since \(E\) is \(G\)-invariant, \(y\in E\) (therefore \(R_{T_{r}}\left(y\right)=\emptyset\) when \(y\notin E\)), and we can then write
\[g=gg_{o}g_{o}^{-1}\in gg_{o}Q_{r}^{\prime}\left(x\right)=Q_{r}^{\prime}\left(gg_{o}.x\right)=Q_{r}^{\prime}\left(y\right),\]
establishing \eqref{eq:tprop}.

All together, from \eqref{eq:dprop} and \eqref{eq:tprop} it readily follows that \(T_{r}\subseteq E\) is a Haar-transient set, with \(r>0\) being arbitrary. Varying \(r\), let us show that \(T_{r}\nearrow E\) as \(r\nearrow+\infty\) modulo \(\mu\). Note that for every \(x\in E\) we have \(\tau_{r}\left(x\right)\searrow 0\) as \(r\nearrow+\infty\), so pick \(r_{x}>0\) sufficiently large with \(\tau_{r_{x}}\left(x\right)<\Psi_{e}\left(x\right)=1\), and therefore \(e\in Q_{r_{x}}^{\prime}\left(x\right)\) and \(x\in T_{r_{x}}\). Finally, while each \(T_{r}\) is only \(\mu\)-measurable, it contains a full-measure Borel subset (see \cite[Theorem (17.10)]{kechris2012descriptive}), and a subset of a Haar-transient set is Haar-transient. It then readily follows that \(E\) is a union of Haar-transients sets modulo \(\mu\), concluding by the recurrence theorem~\ref{thm:recurrence} that \(E\subseteq\mathcal{D}\) modulo \(\mu\).
\end{proof}

\begin{proof}[Proof of \(\left(2\right)\implies\left(3\right)\) in Theorem~\ref{mthm:maharam}]
For \(f\in L_{+}^{1}\left(X,\mu\right)\) define the Borel cocycle
\[\Psi^{f}:G\times X\to\mathbb{R}_{>0},\quad \Psi_{g}^{f}\left(x\right)\coloneqq\varDelta\left(g\right)\nabla_{g}\left(x\right)f\left(g.x\right)/f\left(x\right),\quad \left(g,x\right)\in G\times X,\]
where \(\varDelta\) is the modular function of \(G\) and \(\lambda\). For every \(r>0\) and \(x\in X\) we have
\begin{align*}
\lambda\big(g\in G:\Psi_{g^{-1}}^{f}\left(x\right)\geq r\big)
&=\lambda\left(g\in G:\varDelta\big(g^{-1}\big)\nabla_{g^{-1}}\left(x\right)f\left(g^{-1}.x\right)\geq rf\left(x\right)\right)\\
&=\lambda\left(g\in G:\nabla_{g}\left(x\right)f\left(g.x\right)\geq rf\left(x\right)\right).
\end{align*}
Then for \(\mu\)-a.e. \(x\in X\) there is \(r>0\) with \(\lambda\left(g\in G:\nabla_{g}\left(x\right)f\left(g.x\right)\geq rf\left(x\right)\right)=+\infty\), and therefore \(r_{\left(f,x\right)}\coloneqq rf\left(x\right)\) satisfies the required property.
\end{proof}

\begin{rem}
The same proof with the cocycle \(\left(g,x\right)\mapsto\nabla_{g}\left(x\right)f\left(g.x\right)/f\left(x\right)\) shows that for \(\mu\)-a.e. \(x\in X\) there is \(r>0\) with \(\lambda\left(g\in G:\nabla_{g^{-1}}\left(x\right)f\left(g^{-1}.x\right)\geq r\right)=+\infty\). 
\end{rem}

\begin{proof}[Proof of \(\left(3\right)\implies\left(4\right)\) in Theorem~\ref{mthm:maharam}]
Since conservativity depends only on the measure class, we may assume that \(\mu\) is a probability measure and replace the measure \(\widetilde{\mu}\) on the Maharam extension by the probability measure \(\widehat{\mu}\) given by
\[d\widehat{\mu}\left(x,t\right)=d\mu\left(x\right){\textstyle \frac{1}{2}}e^{-\left|t\right|}dt.\]
Therefore, \(\big(\widetilde{X},\widehat{\mu}\big)\) is no longer infinite measure preserving but a nonsingular probability \(G\)-space, and its associated Radon--Nikodym cocycle \(\widehat{\nabla}\) takes the form
\[\widehat{\nabla}_{g}\left(x,t\right)=\nabla_{g}\left(x\right)e^{\left|t\right|-\left|t+\log\nabla_{g}\left(x\right)\right|},\quad\left(x,t\right)\in\widetilde{X}.\]
By the triangle inequality we have
\[\widehat{\nabla}_{g}\left(x,t\right)\geq\nabla_{g}\left(x\right)e^{-\left|\log\nabla_{g}\left(x\right)\right|}\geq\min\big\{ 1,\nabla_{g}\left(x\right)^{2}\big\},\quad\left(x,t\right)\in\widetilde{X}.\]
Now by the assumption of (3) applied to \(1\in L_{+}^{1}\left(X,\mu\right)\), for \(\mu\)-a.e. \(x\in X\) there is \(0<r_{x}\leq 1\) such that for \(Q\left(x\right)\coloneqq\{g\in G:\nabla_{g}\left(x\right)\geq r_{x}\}\) we have \(\lambda\left(Q\left(x\right)\right)=+\infty\). It follows that for \(\widehat{\mu}\)-a.e. \(\left(x,t\right)\in\widetilde{X}\) (in fact, for \(\mu\)-a.e. \(x\in X\) and every \(t\in\mathbb{R}\)),
\[\int_{G}\widehat{\nabla}_{g}\left(x,t\right)d\lambda\left(g\right)\geq\int_{Q\left(x\right)}\min\big\{1,\nabla_{g}\left(x\right)^{2}\big\} d\lambda\left(g\right)\geq r_{x}^{2}\cdot\lambda\left(Q\left(x\right)\right)=+\infty.\]
We conclude by Theorem~\ref{mthm:hopfdeco} that \(\big(\widetilde{X},\widehat{\mu}\big)\) is conservative, hence so is \(\big(\widetilde{X},\widetilde{\mu}\big)\).
\end{proof}

\begin{proof}[Proof of \(\left(4\right)\implies\left(1\right)\) in Theorem~\ref{mthm:maharam}]
If \(T\subseteq X\) is a transient set of \(\left(X,\mu\right)\) then \(T\times\mathbb{R}\subseteq\widetilde{X}\) is a transient set of \(\big(\widetilde{X},\widetilde{\mu}\big)\), so the proof follows directly from Theorem~\ref{mthm:recurrence}.
\end{proof}

\supersection{The structure of totally dissipative actions}
\label{Section: The Structure of Totally Dissipative Actions}

As noted in the concluding remark of \cite{krengel1968I}, Krengel’s theorem on totally dissipative flows describes their general structure: up to an isomorphism, all totally dissipative nonsingular flows are obtained from some standard measure space \(\left(W_{o},\nu_{o}\right)\) via the construction \eqref{eq:transflow}. Our goal is to establish a structure theorem for totally dissipative nonsingular \(G\)-spaces for a general lcsc group \(G\). As we shall see, when \(G\) admits nontrivial compact subgroups, the general structure is more intricate.

The prototype of a totally dissipative nonsingular \(G\)-space is \(\left(G,\lambda\right)\) itself, which becomes a measure preserving \(G\)-space with the action by left-translations. This extends to two substantially different constructions: \emph{translation \(G\)-spaces} of the form \(\left(W_{o}\times G,\nu_{o}\otimes\lambda\right)\) (see Construction~\ref{cnst:translations}), and \emph{compactly fibered coset \(G\)-spaces} of the form \(\left(G/K,\kappa\right)\) for a compact subgroup \(K\) of \(G\) (see Construction~\ref{cnst:coset}). Those constructions turn out to be prime instances of dissipativity; the former is free and generally non-ergodic (unless \(\left(W_{o},\nu_{o}\right)\) is trivial), and the latter is non-free (unless \(K\) is trivial) and transitive. To obtain a general structure theorem we will unify translation \(G\)-spaces and compactly fibered coset \(G\)-spaces into one construction that we call \emph{Krengel \(G\)-spaces}, as illustrated in Figure~\ref{Figure: structure dissipative}.

\begin{figure}
\captionsetup{font={scriptsize,stretch=1}}
\[
\xymatrix@R=1.7cm@C=0cm{
    & \parbox{3.5cm}{Krengel \(G\)-spaces}\ar@{~>}[dl]|{+\text{free}}\ar@{~>}[dr]|{+\text{transitive}}
    & \\ 
    \parbox{3.7cm}{Translation \(G\)-spaces} \ar@{~>}[dr]|{+\text{transitive}}
    & \boxed{\parbox{4.1cm}{\centering \textbf{Totally dissipative}\\\textbf{nonsingular \(G\)-spaces}}}
    & \parbox{2.8cm}{Compactly fibered coset \(G\)-spaces} \ar@{~>}[dl]|{+\text{free}} \\
    & \qquad\qquad\parbox{3.8cm}{\quad\(G\curvearrowright\left(G,\lambda\right)\)} &
}
\]
\caption[.]{An arrow \(\xymatrix@1{A\ar@{~>}[r]|{+\text{P}}& B}\) presents the passage from class \(A\) of totally dissipative \(G\)-spaces to its subclass \(B\) that satisfies the restrictive property P.}
\label{Figure: structure dissipative}
\end{figure}

\begin{con}
\label{cnst:translations}
A {\bf translation \(G\)-space} is a measure preserving \(G\)-space of the form \(\left(W_{o}\times G,\nu_{o}\otimes\lambda\right)\), for some standard measure space \(\left(W_{o},\nu_{o}\right)\), with the action that is given by
\[g.\left(w,h\right)=\left(w,gh\right),\quad\left(w,h\right)\in W_{o}\times G,\quad g\in G.\]
We will refer to it as the {\bf translation \(G\)-space attached to \(\left(W_{o},\nu_{o}\right)\)}.
\end{con}

It is clear that every translation \(G\)-spaces is free and generally non-ergodic, unless \(\left(W_{o},\nu_{o}\right)\) is trivial (i.e. all Borel sets are null or conull). For free actions, Krengel's structure theorem can be directly generalized:

\renewcommand{\themainthm}{E}
\begin{mainthm}[Following Krengel--Rosinski]
\label{mthm:krengel}
An \emph{essentially free} nonsingular \(G\)-space is totally dissipative iff it is isomorphic to a translation \(G\)-space \eqref{cnst:translations}.
\end{mainthm}
\renewcommand{\themainthm}{\Alph{mainthm}}

To describe the structure of ergodic totally-dissipative \(G\)-spaces, we need the construction of compactly fibered coset \(G\)-spaces, using the well-known construction of measures on homogeneous spaces (see the \hyperref[Appendix: Measures on Homogeneous Spaces]{Appendix}).

\begin{con}
\label{cnst:coset}
A {\bf coset \(G\)-space} is a nonsingular \(G\)-space of the form \(\left(G/K,\kappa\right)\), where \(K\) is a closed subgroup of \(G\), with the action
\[g.hK=ghK,\quad g\in G,\,hK\in G/K,\]
and \(\kappa\) is the unique (up to measure class) quasi-invariant measure. When \(K\) is compact we will call \(\left(G/K,\kappa\right)\) a {\bf compactly fibered coset \(G\)-space}.
\end{con}

For ergodic actions, the following structure theorem appears to be due to Glimm--Effros (see \cite[Definition 5.2]{feldman1978orbit}, \cite[Prop. 2.1.10]{zimmer2013ergodic}, \cite[Thm. 9.12]{nadkarni2013basic}).

\renewcommand{\themainthm}{F}
\begin{mainthm}[Following Glimm--Effros]
\label{mthm:glimmeffros}
An \emph{ergodic} nonsingular \(G\)-space is totally dissipative iff it is isomorphic to a compactly fibered coset \(G\)-space \eqref{cnst:coset}.
\end{mainthm}
\renewcommand{\themainthm}{\Alph{mainthm}}

When \(G\) is countable, a dissipative nonsingular \(G\)-space is never ergodic provided the measure is nonatomic. This is not true in general: coset \(G\)-spaces (which can be nonatomic for uncountable \(G\)) are transitive. Recall that a nonsingular \(G\)-space is {\bf properly ergodic} if it is ergodic and not essentially transitive, i.e. \(\mu\) is not supported on a single orbit. A direct corollary of Theorem~\ref{mthm:glimmeffros} is:

\begin{cor}
Dissipative nonsingular \(G\)-spaces are never properly ergodic.
\end{cor}

We now describe Krengel \(G\)-spaces. Recall the standard Borel spaces
\[\mathbf{K}_{\mathrm{gr}}\left(G\right)\subset\mathbf{K}\left(G\right)\]
of compact subgroups of \(G\) and of compact subsets of \(G\), respectively.

\begin{defn}
A {\bf random compact subgroup} of \(G\) is a standard probability space \(\left(W_{o},\nu_{o}\right)\) together with a Borel map \(\psi_{o}:\left(W_{o},\nu_{o}\right)\to\mathbf{K}_{\mathrm{gr}}\left(G\right)\).
\end{defn}

In defining the Krengel \(G\)-space associated with a random compact subgroup of \(G\), we use that for every compact subgroup \(K\) of \(G\), there is a unique \(G\)-invariant probability measure \(\kappa\) on \(G/K\), obtained as the pushforward of \(\lambda\) from \(G\) to \(G/K\) along the canonical projection, and, moreover, the measure class of this measure is unique (see Theorem~\ref{thm:meashom}(1) and Theorem~\ref{thm:haaruniq} in the \hyperref[Appendix: Measures on Homogeneous Spaces]{Appendix}).

\begin{con}
\label{cnst:krengel}
A {\bf Krengel \(G\)-space} is a nonsingular (probability) \(G\)-space
\[\left(X_{o},\mu_{o}\right)\]
constructed out of random compact subgroup \(\psi_{o}:\left(W_{o},\nu_{o}\right)\to\mathbf{K}_{\mathrm{gr}}\left(G\right)\) as follows:
\begin{itemize}
    \item The Borel \(G\)-space \(X_{o}\) is as in Construction~\ref{cnst:psispace},
    \[X_{o}\coloneqq \left\{\left(x,g\psi_{o}\left(w\right)\right):w\in W_{o},g\in G\right\}.\]
    \item The measure \(\mu_{o}\) is given by
    \[\mu_{o}=\int_{W_{o}}\delta_{\left\{ w\right\} }\otimes\kappa_{w}d\nu_{o}\left(w\right),\]
    where for each \(w\in W_{o}\), \(\kappa_{w}\) is a probability measure in the (unique) measure-class of the \(G\)-invariant measure on \(G/\psi_{o}\left(w\right)\).
\end{itemize}
\end{con}

Then \(\left(X_{o},\mu_{o}\right)\) is a nonsingular \(G\)-space (see Lemma~\ref{lem:standardsigma} below) and it satisfies the following properties:
\begin{enumerate}
    \item \(W_{o}\cong\left\{\left(w,\psi_{o}\left(w\right)\right):w\in W_{o}\right\}\) is a Borel transversal for \(E_{G}^{X_{o}}\).
    \item The stabilizers of points are conjugacy classes of \(\psi_{o}\) (Observation~\ref{obs:stab}).
    \item \(\left(W_{o},\nu_{o}\right)\) is the space of the ergodic components of \(\left(X_{o},\mu_{o}\right)\): a point \(w\in W_{o}\) corresponds to a compactly fibered coset \(G\)-space \(G/\psi_{o}\left(w\right)\).
\end{enumerate}

\renewcommand{\themainthm}{G}
\begin{mainthm}
\label{mthm:generaldiss}
A nonsingular \(G\)-space is totally dissipative iff it is isomorphic to the Krengel \(G\)-space associated with its random compact subgroup.
\end{mainthm}
\renewcommand{\themainthm}{\Alph{mainthm}}

\begin{rem}
\label{rmk:michael}
An alternative approach to define Krengel \(G\)-spaces was suggested to us by Michael Bj\"{o}rklund. For a random compact subgroup \(\psi_{o}:\left(W_{o},\nu_{o}\right)\to\mathbf{K}_{\mathrm{gr}}\left(G\right)\), define on \(W_{o}\times G\) the \(G\)-invariant Borel equivalence relation
\[E_{\psi_{o}}=\big\{ \left(\left(w,g\right),\left(w,g'\right)\right)\in\left(W_{o}\times G\right)^{2}:g\psi_{o}\left(w\right)=g'\psi_{o}\left(w\right)\big\},\]
and let \(\mathcal{I}_{\psi_{o}}\) be the \(\sigma\)-algebra of \(E_{\psi_{o}}\)-invariant sets. By Mackey's point-realization theorem, the Boolean \(G\)-space corresponding to \(\mathcal{I}_{\psi_{o}}\) admits a point-realization, i.e. a nonsingular \(G\)-space \(\left(X_{o},\mu_{o}\right)\) such that \(\mathcal{I}_{\psi_{o}}\) is realized as a factor map 
\[\left(W_{o}\times G,\nu_{o}\otimes\lambda\right)\to\left(X_{o},\mu_{o}\right).\]
It can be shown that \(\left(X_{o},\mu_{o}\right)\) is naturally a Krengel \(G\)-space.
\end{rem}

\section{Free actions: translation spaces}

Here we prove Theorem~\ref{mthm:krengel}.

\begin{thm}
For every essentially free nonsingular \(G\)-space \(\left(X,\mu\right)\), the following are equivalent:
\begin{enumerate}
    \item \(\left(X,\mu\right)\) is totally dissipative.
    \item \(E_{G}^{X}\) is essentially smooth and \(\mu\) admits a \(G\)-invariant equivalent measure.
    \item \(\left(X,\mu\right)\) is isomorphic to a translation \(G\)-space.
\end{enumerate}
\end{thm}

\begin{proof}
(1)\(\implies\)(2). This follows directly from Theorem~\ref{thm:disssmooth} and Proposition~\ref{prop:dissacim}.

(2)\(\implies\)(3). Since \(E_{G}^{X}\) is essentially smooth it has a Borel transversal \(W_{o}\subseteq X\). Consider the restriction of the action map to \(W_{o}\), namely
\[\Phi:W_{o}\times G\to X,\quad \Phi\left(w,g\right)=g.w.\]
Using the transversal property of \(W_{o}\) and that the action of \(G\) is free, \(\Phi\) is injective Borel map, so it admits a Borel inverse that we denote
\[\Phi^{-1}\left(x\right)=\left(\omega\left(x\right),\gamma\left(x\right)\right)\text{ where }\omega:X\to W_{o}\text{ and }\gamma:X\to G.\]
We claim that
\[\omega\left(g.x\right)=\omega\left(x\right)\text{ and }\gamma\left(g.x\right)=g\gamma\left(x\right)\text{ for all }x\in X,\,\,g\in G.\]
Indeed, the first identity is by the transversal property, and this implies
\[g\gamma\left(x\right).\omega\left(x\right)=g.x=\gamma\left(g.x\right).\omega\left(g.x\right)=\gamma\left(g.x\right).\omega\left(x\right),\]
which implies the second identity since the action is free.

Fix some invariant measure \(\eta_{o}\) equivalent to \(\mu\). Consider the measure \(\mu_{1}\coloneqq \eta_{o}\circ\Phi\). Then for every Borel sets \(A\subseteq W_{o}\) and \(B\subseteq G\), and every \(g\in G\),
\begin{align*}
&\mu_{1}\left(A\times g^{-1}B\right)\\
&\quad=\eta_{o}\left(\Phi\left(A\times g^{-1}B\right)\right)\\
&\quad=\eta_{o}\left(x\in X:\omega\left(x\right)\in A,g\omega\left(x\right)\in B\right)\\
&\quad=\eta_{o}\left(x\in X:\omega\left(g.x\right)\in A,\omega\left(g.x\right)\in B\right)\\
&\quad=\eta_{o}\left(x\in X:\omega\left(x\right)\in A,\omega\left(x\right)\in B\right)\\
&=\mu_{1}\left(A\times B\right).
\end{align*}
It follows that for every Borel set \(A\subseteq W_{o}\), the map \(B\mapsto\mu_{o}\left(A\times B\right)\) is a left invariant Borel \(\sigma\)-finite measure on \(G\), so it is \(\lambda\) up to a positive constant (see Theorem~\ref{thm:haaruniq}(2)). Thus, there is a positive constant \(\mu_{o}\left(A\right)\) such that
\[\mu_{1}\left(A\times B\right)=\mu_{o}\left(A\right)\lambda\left(B\right).\]
As \(\mu_{1}\) is a measure, \(A\mapsto\mu_{o}\left(A\right)\) defines a measure on \(W_{o}\) so we deduce (3).

(3)\(\implies\)(1). For a translation \(G\)-space attached to \(\left(W_{o},\nu_{o}\right)\), since it is measure preserving we may assume that \(\nabla\left(w,h\right)=1\) for every \(\left(w,h\right)\in W_{o}\times G\), so for every \(f\in L_{+}^{1}\left(W_{o}\times G,\nu_{o}\otimes\lambda\right)\) we have
\[S_{f}^{G}\left(w,h\right)=\int_{G}f\left(w,gh\right)d\lambda\left(g\right).\]
By the Fubini Theorem and the integrability of \(f\) we have
\begin{align*}
\int_{W_{o}}S_{f}^{G}\left(w,h\right)d\nu_{o}\left(w\right)	
&=\iint\nolimits_{W_{o}\times G}f\left(w,gh\right)d\nu_{o}\otimes\lambda\left(w,h\right)\\
&=\iint\nolimits_{W_{o}\times G}f\left(w,h\right)d\nu_{o}\otimes\lambda\left(w,h\right)<+\infty.
\end{align*}
It follows that \(S_{f}^{G}\left(w,h\right)<+\infty\) for \(\nu_{o}\otimes\lambda\)-a.e. \(\left(w,h\right)\in W_{o}\times G\), concluding that the translation \(G\)-space attached to \(\left(W_{o},\nu_{o}\right)\) is totally dissipative.
\end{proof}

\section{Ergodic actions: compactly fibered coset spaces}

Here we prove Theorem~\ref{mthm:glimmeffros}. Recall that in Construction~\ref{cnst:coset}, the stabilizers of a point \(hK\in G/K\) is the conjugation \(K^{h}=h^{-1}Kh\), thus the action is not essentially free as soon as \(K\) is nontrivial. As for dissipativity we have:

\begin{lem}
\label{lem:cosdiss}
A coset \(G\)-space is totally dissipative iff it is compactly fibered.
\end{lem}

\begin{proof}
In the nonsingular \(G\)-space \(\left(G/K,\kappa\right)\), all the stabilizers are compact exactly when \(K\) is compact. Thus, if \(\left(G/K,\kappa\right)\) is totally dissipative then from Theorem~\ref{thm:compstab} it follows that \(K\) is compact. Conversely, clearly \(E_{G}^{G/K}\) is smooth (it admits a one point Borel transversal), so if \(K\) is compact then every stabilizer is compact, so from Theorem~\ref{mthm:disssmooth} it follows that \(\left(G/K,\kappa\right)\) is totally dissipative.
\end{proof}

We can now prove Theorem~\ref{mthm:glimmeffros}:

\begin{proof}[Proof of Theorem~\ref{mthm:glimmeffros}]
By Lemma~\ref{lem:cosdiss}, a compactly fibered coset \(G\)-space is totally dissipative, so we show the converse. Let \(\left(X,\mu\right)\) be an ergodic totally dissipative nonsingular \(G\)-space, and we start by showing that it is essentially transitive. By Theorem~\ref{thm:smooth}, since \(E_{G}^{X}\) is smooth the space \(X/E_{G}^{X}\) is standard Borel and the quotient map \(\pi:X\rightarrow X/E_{G}^{X}\) is a Borel map. It is then sufficient to show that the pushforward measure \(\pi_{\ast}\mu\) on \(X/E_{G}^{X}\) is a Dirac measure. Indeed, if \(A,B\subseteq X/E_{G}^{X}\) are \(\pi_{\ast}\mu\)-positive disjoint Borel sets, then \(\pi^{-1}\left(A\right),\pi^{-1}\left(B\right)\) are \(\mu\)-positive disjoint Borel sets, and they are clearly \(G\)-invariant, a contradiction to ergodicity.

Now that \(\left(X,\mu\right)\) is essentially transitive, it is isomorphic as a Borel \(G\)-space to an orbit \(\left(G.x_{o},\mu\right)\) for some \(x_{o}\in X\), so the map \(q:G.x_{o}\to G/G_{x_{o}}\), \(q:g.X_{o}\mapsto gG_{x_{o}}\), is an isomorphism of the measure preserving \(G\)-spaces \(\left(G.x_{o},\mu\right)\) and \(\left(G/G_{x_{o}},q_{\ast}\mu\right)\). By Theorem~\ref{thm:compstab}, \(G_{x_{o}}\) is a compact subgroup of \(G\), so by Theorem~\ref{thm:haaruniq}(1) there is a unique measure class of quasi-invariant measures on \(G/G_{x_{o}}\), namely \(q_{\ast}\mu\) is equivalent to the measure \(\kappa\) that forms the compactly fibered coset \(G\)-space \(\left(G/G_{x_{o}},\kappa\right)\). Thus, the nonsingular \(G\)-spaces \(\left(X,\mu\right)\) and \(\left(G/G_{x_{o}},\kappa\right)\) are isomorphic.
\end{proof}

\section{General actions: Krengel spaces}

Here we prove first that Construction~\ref{cnst:krengel} of Krengel \(G\)-spaces is well-founded, and then we prove Theorem~\ref{mthm:generaldiss}. Let us first note that:

\begin{lem}
\label{lem:fibequiv}
Let \(\left(X,\eta\right)\) be a measure space, and let \(\mu,\nu\) be measures on \(X\) such that \(\mu=\int_{X}\mu_{x}d\eta\left(x\right)\) and \(\nu=\int_{X}\nu_{x}d\eta\left(x\right)\). If \(\mu_{x}\sim\nu_{x}\) for \(\eta\)-a.e. \(x\in X\) then \(\mu\sim\nu\).
\end{lem}

\begin{proof}
Let \(A\subseteq X\) be an arbitrary Borel set. Then \(\mu\left(A\right)=0\) iff \(\mu_{x}\left(A\right)=0\) for \(\eta\)-a.e. \(x\in X\), iff \(\nu_{x}\left(A\right)=0\) for \(\eta\)-a.e. \(x\in X\), iff \(\nu\left(A\right)=0\).
\end{proof}

\begin{lem}
\label{lem:standardsigma}
Construction~\ref{cnst:krengel} yields a (standard) nonsingular \(G\)-space.
\end{lem}

\begin{proof}
In Lemma~\ref{lem:standard} we showed that \(X_{o}\) is a standard Borel space. For the nonsingularity, for every \(g\in G\), using that each \(\kappa_{w}\) is quasi-invariant and Lemma~\ref{lem:fibequiv},
\[\mu_{o}\circ g=\int_{W_{o}}\delta_{\left\{ w\right\} }\otimes\left(\kappa_{w}\circ g\right)d\nu_{o}\left(w\right)\sim\int_{W_{o}}\delta_{\left\{ w\right\} }\otimes\kappa_{w}d\nu_{o}\left(w\right)=\mu_{o}.\qedhere\]
\end{proof}

\begin{proof}[Proof of Theorem~\ref{mthm:generaldiss}]
For every Krengel \(G\)-space, by construction all stabilizers are compact and, as demonstrated in Theorem~\ref{thm:smoothactions}, its associated orbit equivalence relation is smooth, and therefore it is totally dissipative. We then prove the converse. Let \(\left(X,\mu\right)\) be a totally dissipative nonsingular \(G\)-space. The stabilizer map \(\psi:X\to\mathbf{K}_{\mathrm{gr}}\left(G\right)\), \(\psi\left(x\right)=G_{x}\), is Borel by Theorem~\ref{thm:miller}, and in light of Theorem~\ref{thm:compstab} its image is indeed in \(\mathbf{K}_{\mathrm{gr}}\left(G\right)\). Using Theorem~\ref{mthm:disssmooth}, there exists a \(\mu\)-conull set \(X_{1}\) such that \(E_{G}^{X_{1}}\) is smooth. By Proposition~\ref{prop:invariance}(2), there is a \(\mu\)-conull set \(X_{2}\subseteq X_{1}\) for which \(G.X_{2}\) is a Borel (\(G\)-invariant) set. Note that any Borel transversal of \(X_{1}\) is then also a Borel transversal of \(G.X_{2}\), thus \(E_{G}^{X_{2}}\) is again smooth. We then assume for simplicity that \(X=X_{2}\) and \(E_{G}^{X}\) admits a Borel transversal. Let the Borel \(G\)-space
\[X_{o}\coloneqq \left\{\left(w,g\psi\left(w\right)\right):w\in W,g\in G\right\},\]
and by Theorem~\ref{thm:smoothactions} there is an isomorphism of Borel \(G\)-spaces
\[\iota:X\to X_{o},\]
satisfying \(\iota\left(G.w\right)=\left\{w\right\}\times G/\psi\left(w\right)\) for every \(w\in W\). We can now define a random compact subgroup \(\psi_{o}:\left(W_{o},\nu_{o}\right)\to\mathbf{K}_{\mathrm{gr}}\left(G\right)\) as follows. First define
\[W_{o}\coloneqq \iota\left(W\right)=\left\{ \left(w,\psi\left(w\right)\right):w\in W\right\},\]
and since \(W_{o}\subseteq X_{o}=\iota\left(X\right)\) forms a Borel transversal, pick a Borel selector \(s:X\to W\). Since we only care of nonsingular isomorphism, we may assume that \(\mu\) is a probability measure, so let \(\nu\coloneqq s_{\ast}\mu\) on \(W\) and define
\[\nu_{o}\coloneqq \iota_{\ast}\nu.\]
Finally, define \(\psi_{o}:W_{o}\to\mathbf{K}_{\mathrm{gr}}\left(G\right)\) by
\[\psi_{o}\left(w,\psi\left(w\right)\right)=\psi\left(w\right),\quad w\in W.\]

Letting \(\left(X_{o},\mu_{o}\right)\) be the associated Krengel \(G\)-space, we will finish the proof by showing that \(\left(X,\mu\right)\) and \(\left(X_{o},\mu_{o}\right)\) are isomorphic as nonsingular \(G\)-spaces, by showing that \(\iota_{\ast}\mu\sim\mu_{o}\). Recall the ergodic decomposition theorem as presented in Section~\ref{sct:ergdecomp}, and write the ergodic decomposition of \(\left(X,\mu\right)\) as
\[\mu=\int_{X}\mu_{x}d\mu\left(x\right)=\int_{X}\mu_{s\left(x\right)}d\mu\left(x\right)=\int_{W}\mu_{w}d\nu\left(w\right),\]
where we used that \(x\mapsto\mu_{x}\) is a \(G\)-invariant map hence constant along orbits. For each \(w\in W\), the ergodic component \(\left(X,\mu_{w}\right)\) is a totally dissipative ergodic nonsingular \(G\)-space, hence by Theorem~\ref{mthm:glimmeffros} it is isomorphic to the compactly fibered \(G\)-space \(\left(G/\psi\left(w\right),\kappa_{w}\right)\), thus \(\mu_{w}\sim\kappa_{w}\). Since this holds for every \(w\in W\), from Lemma~\ref{lem:fibequiv} it follows that
\[\mu=\int_{W}\mu_{w}d\nu\left(w\right)\sim\int_{W}\kappa_{w}d\nu\left(w\right).\]
Noting that \(\iota_{\ast}\kappa_{w}=\delta_{\left\{ w\right\} }\otimes\kappa_{w}\) for every \(w\in W\), we obtain
\[\iota_{\ast}\mu\sim\int_{W}\iota_{\ast}\kappa_{w}d\nu\left(w\right)=\int_{W_{o}}\delta_{\left\{ w\right\} }\otimes\kappa_{w}d\nu_{o}\left(w\right)=\mu_{o}.\qedhere\]
\end{proof}

\supersection{Further properties}

\section{Maharam's recurrence theorem}

The following theorem was proved in \cite[Lemma 5.5]{maharam1964incompressible} for a single transformation (see \cite[Theorem 1.1.7]{aaronson1997introduction}) and in \cite[Proposition 1.2]{kaimanovich1994ergodicity} for countable groups.

\begin{thm}[Maharam's recurrence theorem]
\label{thm:maharamrecurrence}
Let \(G\) be an lcsc non-compact group and \(\left(X,\mu\right)\) a measure preserving \(G\)-space. If there is a Borel set \(A_{o}\subseteq X\) with
\[0<\mu\left(A_{o}\right)<+\infty\text{ and }\lambda\left(R_{A_{o}}\left(x\right)\right)=+\infty\text{ for }\mu\text{-a.e. }x\in X,\]
then \(\left(X,\mu\right)\) is conservative.
\end{thm}

\begin{proof}
Suppose that \(T\subseteq X\) is a Haar-transient set. With the sets
\[T_{r}\coloneqq \left\{ x\in T:\lambda\left(R_{T}\left(x\right)\right)\leq r\right\} ,\quad r>0,\]
as in the proof of Proposition~\ref{prop:recurrence}, the computation \eqref{eq:maharam} for \(f=1_{A_{o}}\) (which holds true for general nonnegative functions) shows that for every \(r>0\),
\[\int_{T_{r}}S_{1_{A_{o}}}^{G}\left(x\right)d\mu\left(x\right)\leq r\mu\left(A_{o}\right)<+\infty\text{, so }S_{1_{A_{o}}}^{G}\left(x\right)<+\infty\text{ for }\mu\text{-a.e. }x\in T_{r}.\]
However, as we are in the measure preserving case and by the assumption,
\[S_{1_{A_{o}}}^{G}\left(x\right)=\lambda\left(R_{A_{o}}\left(x\right)\right)=+\infty,\]
hence \(\mu\left(T_{r}\right)=0\). Since \(T_{r}\nearrow T\) as \(r\nearrow +\infty\) we deduce that \(\mu\left(T\right)=0\). Then by the recurrence theorem~\ref{mthm:recurrence}, \(\left(X,\mu\right)\) is conservative.
\end{proof}

\section{Maximal transient set}

For a countable group \(G\), it is known that there exists a wandering set \(W_{\max}\subseteq X\) with \(\mathcal{D}=G.W_{\max}\) modulo \(\mu\) (see \cite[Proposition 1.6.1]{aaronson1997introduction}). The following is the analog for transient sets, and its proof is essentially due to Aaronson \cite[Proposition 1.6.1]{aaronson1997introduction}.

\begin{prop}
Let \(\left(X,\mu\right)\) be a nonsingular \(G\)-space. There exists a Haar-transient set \(T_{\max}\subseteq X\) with \(\mathcal{D}=G.T_{\max}\) modulo \(\mu\). Moreover, when \(\mu\left(\mathcal{D}\right)>0\) one may assume further that \(\mu\left(T_{\max}\right)>0\).
\end{prop}

\begin{proof}
We may assume that \(\mu\) is a probability measure. Recall the hereditary collection \(\mathfrak{T}\) of Haar-transient sets of \(\left(X,\mu\right)\). Let us proceed in an inductive way as follows. First, if \(\mu\left(\mathcal{D}\right)=0\) set \(T_{\max}\coloneqq \emptyset\). Otherwise, set
\[\alpha_{1}\coloneqq \sup\left\{ \mu\left(T\right):T\in\mathfrak{T}\right\}>0,\]
and pick \(T_{1}\in\mathfrak{T}\) with \(\mu\left(T_{1}\right)>\alpha_{1}/2\). Suppose \(T_{1},\dots,T_{n}\in\mathfrak{T}\) were defined. If \(\mu\left(\mathcal{D}\setminus G.\left(T_{1}\sqcup\dotsm\sqcup T_{n}\right)\right)=0\) set \(T_{\max}=:T_{1}\sqcup\dotsm\sqcup T_{n}\). Otherwise, set
\[\alpha_{n}\coloneqq \sup\left\{\mu(T):T\in\mathfrak{T},T\subseteq\mathcal{D}\setminus G.\left(T_{1}\sqcup...\sqcup T_{n}\right)\right\}.\]
Either this process halts, and then \(T_{\mathrm{max}}\) is defined, or that it is indefinite, in which case set \(T_{\max}\coloneqq T_{1}\sqcup T_{2}\sqcup\dotsm\). We claim that \(T_{\max}\) is a maximal Haar-transient set. For every \(x\in X\), either \(R_{T_{\max}}\left(x\right)=\emptyset\) or \(R_{T_{\max}}\left(x\right)=R_{T_{n_{x}}}\left(x\right)\) for some \(n_{x}\in\mathbb{N}\), hence \(\lambda\left(R_{T_{\max}}\left(x\right)\right)<+\infty\) and \(T_{\max}\in\mathfrak{T}\). We note that for every \(g\in G\), since \(\mu\left(T_{\max}\right)>0\) also \(\mu\left(g.T_{\max}\right)>0\), and since \(g.T_{\max}\in\mathfrak{T}\) we have \(g.T_{\max}\subseteq\mathcal{D}\) modulo \(\mu\). This implies that \(S.T_{\max}\subseteq\mathcal{D}\) modulo \(\mu\) for every countable set \(S\subset G\). By the continuity of the mapping \(G\to\left[0,1\right]\), \(g\mapsto\mu\left(g.T_{\mathrm{max}}\right)\), the separability of \(G\) and using that \(\mu\left(T_{\max}\right)>0\), we deduce that \(G.T_{\max}\subseteq\mathcal{D}\) modulo \(\mu\). Finally, since \(T_{1},T_{2},\dotsc\) are disjoint, we have
\[\sum\nolimits_{n\in\mathbb{N}}\alpha_{n}=2\cdot\sum\nolimits_{n\in\mathbb{N}}\mu\left(T_{n}\right)\leq2\cdot\mu\left(X\right)<+\infty,\]
so we have \(\mu\left(\mathcal{D}\backslash G.T_{\max}\right)\leq\alpha_{n}\xrightarrow[n\to\infty]{}0\), thus \(\mathcal{D}=G.T_{\max}\) modulo \(\mu\).
\end{proof}

\section{The ergodic decomposition}
\label{sct:ergdecomp}

Another consequence of the recurrence theorem~\ref{mthm:recurrence} is that dissipativity is preserved when passing to ergodic components. For a nonsingular \(G\)-space \(\left(X,\mu\right)\) with Radon--Nikodym cocycle \(\nabla\), let 
\[\mathcal{E}_{\nabla}^{G}\left(X\right)\]
be the space of measures on \(X\), with respect to which the action of \(G\) is nonsingular and ergodic, and whose Radon--Nikodym cocycle is \(\nabla\) almost everywhere. Thus, \(\upsilon\in\mathcal{E}_{\nabla}^{G}\left(X\right)\) when \(\left(X,\upsilon\right)\) is an ergodic nonsingular \(G\)-space and
\[\frac{d\upsilon\circ g}{d\upsilon}=\nabla_{g}\quad\mu\text{-a.e. for every }g\in G.\]
By the ergodic decomposition theorem (see \cite[Theorem 1.1 \& 5.2]{greschonig2000ergodic}),\footnote{While this is formulated for nonsingular probability \(G\)-spaces, one can derive the general case by passing to an equivalent probability measure.} there is a \(G\)-invariant Borel map \(X\to\mathcal{E}_{\nabla}^{G}\left(X\right)\), \(x\mapsto\mu_{x}\), such that the measures \(\{\mu_{x}:x\in X\}\) are pairwise singular, and for every Borel set \(A\subseteq X\), the map
\[X\to\left[0,+\infty\right],\quad x\mapsto\mu_{x}\left(A\right),\]
is Borel, and 
\[\mu\left(A\right)=\int_{X}\mu_{x}\left(A\right)d\mu\left(x\right).\]
The elements of \(\left\{\mu_{x}:x\in X\right\}\) are referred to as the {\bf ergodic components} of \(\left(X,\mu\right)\).

\begin{prop}
\label{prop:dissergodicdeco}
A nonsingular \(G\)-space is totally dissipative iff essentially all of its ergodic components are totally dissipative.
\end{prop}

\begin{proof}
For an arbitrary Borel set \(A\subseteq X\) let
\[A_{\infty}\coloneqq \left\{x\in X:\lambda\left(R_{A}\left(x\right)\right)=+\infty\right\}.\]
With the ergodic decomposition of \(\left(X,\mu\right)\), we have \(\mu\left(A_{\infty}\right)=0\) iff \(\mu_{x}\left(A_{\infty}\right)=0\) for \(\mu\)-a.e. \(x\in X\), so the proposition follows from the recurrence theorem~\ref{mthm:recurrence}.
\end{proof}

\section{Subactions of closed subgroups}

The forthcoming proposition determines the Hopf decomposition for the action of a closed subgroup of the acting group. This was first done by Halmos, who showed that the Hopf decomposition of a single transformation \(S\) coincides with that of its self-iteration \(S^{p}\), for any \(p\in\mathbb{N}\), and then Krengel observed the analog for lattices \cite[Lemma 2.7]{krengel1969II} (see \cite[Corollary 1.1.4 \& Theorem 1.6.4]{aaronson1997introduction}). The following proposition is the same as \cite[Proposition 4.36]{arano2021ergodic}.

\begin{prop}
\label{prop:subgroup}
Let \(\left(X,\mu\right)\) be a nonsingular \(G\)-space. For every closed subgroup \(H\) of \(G\) the following hold.
\begin{enumerate}
    \item If \(\left(X,\mu\right)\) is totally dissipative, then it is also totally dissipative with respect to \(H\) (in the subaction of \(H\)).
    \item If \(H\) is further a lattice, then the converse of (1) is also true, and the Hopf decomposition with respect to \(G\) coincides with the one with respect to \(H\).
\end{enumerate}
\end{prop}

\begin{proof}
Part (1). By Proposition~\ref{prop:dissergodicdeco} it is sufficient to show that every ergodic component of \(\left(X,\mu\right)\) is dissipative as a nonsingular \(H\)-space, and by Theorem~\ref{mthm:glimmeffros} every such ergodic component is isomorphic to a compactly fibered coset \(G\)-spaces \(\left(G/K,\kappa\right)\). Note that the stabilizers of the action of \(H\) on \(G/K\) are of the form \(H\cap K^{g}\) for some \(g\in G\), hence are all compact. Also, by Theorem~\ref{thm:compstabsmooth} in the \hyperref[Appendix: Measures on Homogeneous Spaces]{Appendix}, the orbit equivalence relation \(E_{G/K}^{H}\) is smooth. Then from Theorem~\ref{mthm:disssmooth} it follows that \(\left(G/K,\kappa\right)\) is totally dissipative also as a nonsingular \(H\)-space.

Part (2). When \(H\) is a lattice we may fix a fundamental domain \(\Omega\subset G\) for \(H\), so that \(0<\lambda\left(\Omega\right)<+\infty\) and the translations \(\left\{\omega\Omega:\omega\in H\right\}\) of \(\Omega\) are pairwise disjoint and \(G=\bigcup_{\omega\in H}\omega\Omega\). Fix any \(f\in L_{+}^{1}\left(X,\mu\right)\) and let
\[f_{\Omega}:X\to\mathbb{R}_{>0},\quad f_{\Omega}\left(x\right)\coloneqq \int_{\Omega}\nabla_{g}\left(x\right)f\left(g.x\right)d\lambda\left(g\right).\]
Note that \(f_{\Omega}\in L_{+}^{1}\left(X,\mu\right)\) since \(\lambda\left(\Omega\right)<+\infty\). We now have the relation
\begin{align*}
S_{f}^{G}\left(x\right)
&=\sum\nolimits_{\omega \in H}\int_{\omega\Omega}\nabla_{g}\left(x\right)f\left(g.x\right)d\lambda\left(g\right)\\
&=\sum\nolimits_{\omega \in H}\int_{\Omega}\nabla_{g\omega}\left(x\right)f\left(g\omega.x\right)d\lambda\left(g\right)\\
&=\sum\nolimits_{\omega \in H}\nabla_{\omega}\left(x\right)\int_{\Omega}\nabla_{g}\left(\omega.x\right)f\left(g\omega.x\right)d\lambda\left(g\right)\\
&=\sum\nolimits_{\omega \in H}\nabla_{\omega}\left(x\right)f_{\Omega}\left(\omega.x\right)=S_{f_{\Omega}}^{H}\left(x\right),
\end{align*}
where the last equality is nothing but the very definition of the averaging transform \(S^{H}\) for the discrete countable acting group \(H\) with its Haar (counting) measure. In view of Theorem~\ref{mthm:hopfdeco}, this identity completes the proof.
\end{proof}

\section{Invariant measures and the positive--null decomposition}

Here we deal with another general decomposition of nonsingular \(G\)-spaces, the \emph{positive--null decomposition}, regarding whether a nonsingular \(G\)-space \(\left(X,\mu\right)\) admits a \(G\)-invariant measure absolutely continuous with respect to \(\mu\). This problem is fundamental in ergodic theory and it is of great interest already for a single transformation (see the bibliographical notes in \cite[\S3.4, pp. 144--146, \S6.3, pp. 220--221]{krengel1985ergodic}, the many references in \cite{danilenko2023ergodic}, and \cite[\S10]{nadkarni2013basic}).

\begin{defn}
Let \(\left(X,\mu\right)\) be a nonsingular \(G\)-space. A Borel \(\sigma\)-finite measure is called {\sc a.c.i.m} of \(\left(X,\mu\right)\) if it is both \(G\)-invariant and absolutely continuous with respect to \(\mu\). A probability {\sc a.c.i.m} will be called  {\sc a.c.i.p}.
\end{defn}

We emphasize that an {\sc a.c.i.m} of \(\left(X,\mu\right)\) need not be equivalent to \(\mu\) but merely absolutely continuous with respect to \(\mu\). We also stress that an {\sc a.c.i.m} is finite or infinite, thus every {\sc a.c.i.p} is in particular an {\sc a.c.i.m}. We follow the convention that all finite measures are normalized to be probability measures. The starting point is the following classical theorem which can be proved in many ways (it is a particular case of Theorem~\ref{thm:maharamrecurrence}).

\begin{thm}[Poincar\'{e} recurrence theorem]
\label{thm:poincareacip}
Let \(G\) be an lcsc non-compact group. Then every totally dissipative nonsingular \(G\)-space admits no {\sc a.c.i.p}.
\end{thm}

\begin{prop}
\label{prop:dissacim}
For every nonsingular \(G\)-space \(\left(X,\mu\right)\), the dissipative part \(\mathcal{D}\) admits an {\sc a.c.i.m} equivalent to \(\mu\mid_{\mathcal{D}}\).
\end{prop}

\begin{proof}
Fix any \(f_{0}\in L_{+}^{1}\left(X,\mu\right)\) and define the function
\[f_{0}^{\ast}:X\to\mathbb{R}_{>0},\quad f_{0}^{\ast}\left(x\right)\coloneqq S_{f_{0}}^{G}\left(x\right)=\int_{G}\nabla_{g}\left(x\right)f_{0}\left(g.x\right)d\lambda\left(g\right).\]
Assuming for simplicity that \(X=\mathcal{D}\), we have \(f_{0}^{\ast}<+\infty\) on a \(\mu\)-conull set, so define a measure \(\eta_{0}\) on \(X\) by
\[d\eta_{0}\left(x\right)=f_{0}^{\ast}\left(x\right)d\mu\left(x\right).\]
Since \(f_{0}^{\ast}\) is positive on a \(\mu\)-conull set, \(\eta_{0}\) is equivalent to \(\mu\). To show that \(\eta_{0}\) is \(G\)-invariant, note that by the formula \eqref{eq:nabla}, for every Borel function \(f_{1}:X\to\mathbb{R}_{\geq 0}\),
\begin{align*}
\int_{X}f_{1}\left(x\right)d\eta_{0}\left(x\right)
&=\iint\nolimits_{G\times X}\nabla_{g}\left(x\right)f_{0}\left(g.x\right)f_{1}\left(x\right)d\lambda\otimes\mu\left(g,x\right)\\
&=\iint\nolimits_{G\times X}f_{0}\left(x\right)f_{1}\left(g^{-1}.x\right)d\lambda\otimes\mu\left(g,x\right).
\end{align*}
In then follows by the invariance of the Haar measure together with the Fubini Theorem that, for an arbitrary \(h\in G\), if we replace \(f_{1}\) by \(f_{1}\circ h\) this integral remains unchanged, namely \(\eta_{0}\) is \(G\)-invariant.
\end{proof}

The Poincar\'{e} recurrence theorem~\ref{thm:poincareacip} asserts that only the conservative part may support an {\sc a.c.i.p}. It then can be one's goal to decompose the conservative part further according to the existence of an {\sc a.c.i.p}. The positive--null decomposition for \(G=\mathbb{Z}^{d}\) is classical \cite[\S3.4, Theorem 4.6, \S6.3, Theorem 3.9]{krengel1985ergodic}, and for general lcsc groups it can be derived using the machinery of {\it weakly wandering functions} (see \cite[\S3.4; \S3.5 Theorem 4.9; \S6.3 Theorem 3.10]{krengel1985ergodic}, \cite{grabarnik1995singular}, \cite[\S4]{avraham2023symmetric}). The following is a refined version that takes into account also infinite measures.

\begin{thm}[Positive--null decomposition]
\label{mthm:posnulldeco}
Let \(\left(X,\mu\right)\) be a nonsingular \(G\)-space. There exists an essentially unique \(G\)-invariant decomposition
\[X=\mathcal{P}_{1}\sqcup\mathcal{P}_{\infty}\sqcup\mathcal{N},\]
such that:
\begin{enumerate}
    \item \(\left(\mathcal{P}_{1},\mu\mid_{\mathcal{P}_{1}}\right)\), unless \(\mu\)-null, admits an equivalent {\sc a.c.i.p}.
    \item \(\left(\mathcal{P}_{\infty},\mu\mid_{\mathcal{P}_{\infty}}\right)\), unless \(\mu\)-null, admits an equivalent {\sc a.c.i.m} but no {\sc a.c.i.p}.
    \item \(\left(\mathcal{N},\mu\mid_{\mathcal{N}}\right)\) admits no {\sc a.c.i.m}.
\end{enumerate}
\end{thm}

It is important to note that \(\mathcal{P}_{1}\) may support an infinite {\sc a.c.i.m} in addition to an {\sc a.c.i.p}. Thus, the part
\[\mathcal{P}\coloneqq \mathcal{P}_{1}\sqcup\mathcal{P}_{\infty},\]
ought to be the maximal part of \(\left(X,\mu\right)\) that supports an {\sc a.c.i.m} (finite or infinite). Our approach in proving Theorem~\ref{mthm:posnulldeco} is the following: in the first stage we decompose \(X=\mathcal{P}\sqcup\mathcal{N}\), with the property that \(\mathcal{P}\) supports an {\sc a.c.i.m} and \(\mathcal{N}\) does not support an {\sc a.c.i.m}, and in the second stage we decompose \(\mathcal{P}=\mathcal{P}_{1}\sqcup\mathcal{P}_{\infty}\), with the property that \(\mathcal{P}_{1}\) fully supports an {\sc a.c.i.p} and \(\mathcal{P}_{\infty}\) does not support an {\sc a.c.i.p}. The relations of the positive--null decomposition with the Hopf decomposition are illustrated in Figure~\ref{Figure: decompositions}.

\begin{figure}\captionsetup{font={scriptsize,stretch=1}}
    \centering
    \begin{tikzpicture}
        \pgfmathsetseed{71}

        \node[randomcircle=5cm](A) at (0,0){};

        \path[name path=hline] (A.0) -- (A.200);
        \path[name path=curve] (A.10) .. controls +(170:1) and +(120:3) .. (A.-60);

        \path[name intersections={of=hline and curve, by=I}];

        \draw[randomline, name intersections={of=hline and curve, by=I}] (A.200) -- (I);
        \draw[randomline] (A.120) -- (I);
        \draw[randomline] (A.10) .. controls +(205:1) and +(115:3) .. (A.-60);

        \node at (-1.3, 0.5) {\(\mathcal{C}\cap\mathcal{P}_{1}\)};
        \node at (0.9, 1.2) {\(\mathcal{C}\cap\mathcal{P}_{\infty}\)};
        \node at (1.4, -0.8) {\(\mathcal{D}=\mathcal{D}\cap\mathcal{P}_{\infty}\)};
        \node at (-0.6, -1.3) {\(\mathcal{C}\cap\mathcal{N}\)};

        \end{tikzpicture}

    \caption[.]{The relations between the Hopf decomposition \(X=\mathcal{C}\sqcup\mathcal{D}\) and the positive--null decomposition \(X=\mathcal{P}_{1}\sqcup\mathcal{P}_{\infty}\sqcup\mathcal{N}\).\\ 
    Unless \(G\) is compact, by the Poincar\'{e} recurrence theorem \ref{thm:poincareacip} and by Proposition~\ref{prop:dissacim}, we have \(\mathcal{D}\subseteq\mathcal{P}_\infty\).}
\label{Figure: decompositions}
\end{figure}

The classical Neveu decomposition deals with {\sc a.c.i.p}'s and thus, in our terminology, it refers to the coarser decomposition
\[X=\mathcal{P}_{1}\sqcup X\backslash\mathcal{P}_{1}.\]
Hajian and Ito characterized the Neveu decomposition as follows. A Borel set \(W\) is {\bf weakly wandering} if there exists an infinite set \(S\subseteq G\) such that \(\left\{s.W:s\in S\right\}\) are pairwise disjoint. The collection of weakly wandering sets is hereditary, so the following theorem follows from \cite[Theorem 1]{hajian1969weakly} with the exhaustion lemma~\ref{lem:exhaust}.

\begin{thm}[Hajian--Ito]
Let \(\left(X,\mu\right)\) be a nonsingular \(G\)-space. The decomposition \(X=\mathcal{P}_{1}\sqcup\left(X\backslash\mathcal{P}_{1}\right)\) is determined by:
\begin{enumerate}
    \item \(\mathcal{P}_{1}\) contains no weakly wandering set modulo \(\mu\).
    \item Every \(\mu\)-positive set in \(X\backslash\mathcal{P}_{1}\), if exists at all, contains a \(\mu\)-positive weakly wandering subset.
\end{enumerate}
\end{thm}

Toward proving Theorem~\ref{mthm:posnulldeco} we establish some convenient terminology.

\begin{defn}
For an {\sc a.c.i.m} \(\nu\) of a nonsingular \(G\)-space \(\left(X,\mu\right)\) define:
\begin{description}
    \item[\(\bullet\) support] A Borel set \(P\subseteq X\) with \(\nu\mid_{P}=\nu\) and \(\nu\mid_{X\backslash P}=0\).
    \item[\(\bullet\) compatible support] A \(G\)-invariant support \(P\) with \(\mu\mid_{P}\sim\nu\mid_{P}\).
    \item[\(\bullet\) maximal compatible support] a compatible support that contains, modulo \(\mu\), the compatible support of any other {\sc a.c.i.m}.
    \item[\(\bullet\) \(1\)-maximal compatible support] a compatible support that contains, modulo \(\mu\), the compatible support of any other {\sc a.c.i.p}.
\end{description}
\end{defn}

\begin{lem}
\label{lem:compsupp}
Every {\sc a.c.i.m} admits a compatible support.
\end{lem}

\begin{proof}
For a given {\sc a.c.i.m} \(\nu\), pick a version of the Radon--Nikodym derivative \(d\nu/d\mu\in L^{1}\left(X,\mu\right)\), i.e. a function \(\phi:X\to\mathbb{R}\) with \(\phi=d\nu/d\mu\) \(\mu\)-a.e. and set
\[P\coloneqq \left\{x\in X:\phi\left(x\right)>0\right\}.\]
Clearly \(P\) supports \(\nu\). We also have that \(\mu\mid_{P}\sim\nu\mid_{P}\): clearly \(\nu\mid_{P}\ll\mu\mid_{P}\), and also \(\mu\mid_{P}\ll\nu\mid_{P}\) since if \(\nu\left(E\right)=0\) then
\[\mu\left(E\cap P\right)=\int_{E\cap P}1/\phi d\nu=0.\]
We show that \(P\) is \(\mu\)-almost \(G\)-invariant. Passing to a probability measure equivalent to \(\mu\) if necessary, we may assume that \(\mu\) is a probability measure. Let \(g\in G\) be arbitrary. On one hand we have
\[\mu\left(g.P\cap P\right)=\int_{g.P}1_{P}\cdot1/\phi d\nu=\int_{g.P}1/\phi d\nu=\mu\left(g.P\right),\]
implying that \(\mu\left(g.P\backslash P\right)=0\). On other other hand, since \(\nu\) is \(G\)-invariant we have
\[\nu\left(P\backslash g.P\right)=\nu\left(g^{-1}.P\backslash P\right)=0,\]
and using that \(\mu\mid_{P}\sim\nu\mid_{P}\) we deduce that \(\mu\left(P\backslash g.P\right)=0\). All together we deduce that \(\mu\left(g.P\triangle P\right)=0\), and by Proposition~\ref{prop:invariance}(1) and using that \(\nu\ll\mu\) we may assume that \(P\) is \(G\)-invariant.
\end{proof}

\begin{lem}
\label{lem:maxacip}
Unless an {\sc a.c.i.m} does not exist, there exists an {\sc a.c.i.m} (which is possibly an {\sc a.c.i.p}) with a maximal compatible support.\\
Similarly, unless {\sc a.c.i.p} does not exist, there exists an {\sc a.c.i.p} with a \(1\)-maximal compatible support.
\end{lem}

The following is a proof of the first statement of Lemma~\ref{lem:maxacip}, with the proof of the second statement following similarly, requiring only minor modifications.

\begin{proof}
Passing to a probability measure equivalent to \(\mu\) if necessary, we may assume that \(\mu\) is a probability measure. Set
\[\alpha\coloneqq \sup\left\{\mu\left(P\right):P\text{ is a compatible support of an {\sc a.c.i.m} of }\left(X,\mu\right)\right\}.\] 
Choose a sequence of {\sc a.c.i.m}'s \(\nu_{1},\nu_{2},\dotsc\) with corresponding compatible supports \(P_{1},P_{2},\dotsc\) such that \(\mu\left(P_{n}\right)\to\alpha\) as \(n\to +\infty\). Observe that finite convex combinations of {\sc a.c.i.m}'s is again an {\sc a.c.i.m}, with the union of the underlying compatible supports being a compatible support for their convex combination, so we may assume that \(P_{1}\subseteq P_{2}\subseteq\dotsm\). Set \(P_{0}=\emptyset\) and define a measure \(\upsilon\) by
\[\upsilon\mid_{P_{n}\backslash P_{n-1}}=\nu_{n}/2^{n+1},\quad n\in\mathbb{N}.\]
Is is easy to see that \(\upsilon\) is \(\sigma\)-finite (possibly finite) and, since \(\nu_{n}\) is \(G\)-invariant and \(P_{n}\) is a \(G\)-invariant set for every \(n\in\mathbb{N}\), we also see that \(\upsilon\) is \(G\)-invariant. Thus, \(\upsilon\) is an {\sc a.c.i.m}. Finally, we claim that
\[\mathcal{P}\coloneqq P_{1}\cup P_{2}\cup\dotsm\]
serves as the desired maximal compatible support; it is clearly a compatible support for \(\upsilon\) and, since \(\mu\left(\mathcal{P}\right)\geq\mu\left(P_{n}\right)\) for every \(n\) necessarily \(\mu\left(\mathcal{P}\right)=\alpha\), so for any compatible support \(P\) of any {\sc a.c.i.m} \(\nu\), the set \(\mathcal{P}\cup P\) forms a compatible support for the {\sc a.c.i.m} \(\upsilon/2+\nu/2\), so that \(\alpha=\mu\left(P_{0}\right)\leq\mu\left(P_{0}\cup P\right)\leq\alpha\) and hence \(\mu\left(P\backslash P_{0}\right)=0\), thus \(\mathcal{P}\) is maximal.
\end{proof}

\begin{proof}[Proof of Theorem~\ref{mthm:posnulldeco}]
If there is no {\sc a.c.i.m} we put \(\mathcal{N}\coloneqq X\). Otherwise, using the first statement of Lemma~\ref{lem:maxacip}, fix an {\sc a.c.i.m} \(\upsilon\) with a maximal compatible support \(\mathcal{P}\), let \(\mathcal{N}\coloneqq X\backslash\mathcal{P}\) and look at the \(G\)-invariant decomposition
\[X=\mathcal{P}\sqcup\mathcal{N}.\]
Since \(\mathcal{P}\) is a compatible support, \(\upsilon\mid_{\mathcal{P}}\sim\mu\mid_{\mathcal{P}}\). From the maximality of \(\mathcal{P}\), the compatible support of every {\sc a.c.i.m}, which always exists by Lemma~\ref{lem:compsupp}, must be contained in \(\mathcal{P}\) modulo \(\mu\), hence \(\mathcal{N}\) supports no {\sc a.c.i.m}. This establishes the property (3) in the theorem for the set \(\mathcal{N}\).

We now restrict our attention to the nonsingular \(G\)-space \(\left(\mathcal{P},\mu\mid_{\mathcal{P}}\right)\). If there is no {\sc a.c.i.p} of \(\left(\mathcal{P},\mu\mid_{\mathcal{P}}\right)\) we put \(\mathcal{P}_{\infty}=\mathcal{P}\). Otherwise, using the second statement of Lemma~\ref{lem:maxacip}, fix an {\sc a.c.i.p} \(\varrho\) with a \(1\)-maximal compatible support \(\mathcal{P}_{1}\), let \(\mathcal{P}_{\infty}=\mathcal{P}\backslash\mathcal{P}_{1}\) and look at the \(G\)-invariant decomposition
\[\mathcal{P}=\mathcal{P}_{1}\sqcup\mathcal{P}_{\infty}.\]
The very same reasoning as in the first part of this proof establishes properties (1) and (2) in the theorem for the sets \(\mathcal{P}_{1}\) and \(\mathcal{P}_{\infty}\).

Finally, the maximality of \(\mathcal{P}\) as the compatible support of the {\sc a.c.i.m} \(\upsilon\) together with the \(1\)-maximality of \(\mathcal{P}_{1}\) as the compatible support of the {\sc a.c.i.p} \(\varrho\), imply the uniqueness of the decomposition.
\end{proof}

\section{Concluding remarks}

\subsubsection{Shelah--Weiss characterization of transience}

Shelah and Weiss \cite{shelah1982measurable} showed that for an aperiodic Borel automorphism \(S\) of a standard Borel space \(X\), every Borel set \(A\subseteq X\) that satisfies \(\mu\left(A\right)=0\) for every \(S\)-nonsingular \(S\)-ergodic measure \(\mu\) on \(X\), is contained in \(\bigcup_{n\in\mathbb{Z}}S^{n}\left(W\right)\) for some wandering set \(W\subseteq X\) (note that the conclusion is measure-free). We ask for the following continuous analog:

\begin{que}
Let \(G\) be an lcsc non-compact group with a left Haar measure \(\lambda\) and \(X\) a free Borel \(G\)-space. Suppose \(A\subseteq X\) is a Borel set with \(\lambda\left(R_{A}\left(x\right)\right)>0\) for every \(x\in G.A\), such that \(\mu\left(A\right)=0\) for every quasi-invariant properly ergodic measure \(\mu\) on \(X\). Is it true that \(A\subseteq G.T\) for some transient set \(T\subseteq X\)?
\end{que}

Note that there is always a lacunary cross-section \(C\), that is a Borel set \(C\subset X\) satisfying \(G.C=X\) and there is an identity neighborhood \(U\subset G\) such that \(U.x\cap U.y=\emptyset\) for all distinct \(x,y\in C\) (see e.g. \cite[\S2]{slutsky2017lebesgue} and the references therein). Therefore, the assumption \(\lambda\left(R_{A}\left(x\right)\right)>0\) whenever \(x\in G.A\) is necessary.

\subsubsection{First cohomology, skew products, and Maharam extensions}

Schmidt’s monograph~\cite{schmidt1977cocycles} offers a comprehensive study of the cohomology of nonsingular \(G\)-spaces for countable groups \(G\) (see also~\cite[\S8]{aaronson1997introduction}). Although we do not develop this theory here, we believe it should be formulated and established in full generality. This would involve properly defining the first cohomology~\cite[\S2]{schmidt1977cocycles}, essential values~\cite[\S3]{schmidt1977cocycles}, and recurrence of cocycles~\cite[Definition~3.13]{schmidt1977cocycles}. This last notion should correspond to the conservativity of the skew-product defined by a cocycle~\cite[\S5]{schmidt1977cocycles}, and thus, by the Maharam's theorem~\ref{mthm:maharam}, the Radon--Nikodym cocycle of every conservative nonsingular \(G\)-space is recurrent~\cite[\S4]{schmidt1977cocycles}.

\section*{Appendix: measures on homogeneous spaces}
\label{Appendix: Measures on Homogeneous Spaces}

The construction of measures on homogeneous spaces was established by Mackey \cite[\S1]{mackey1952induced} and extended later by many others, and we recall its basics here in a way that suits our needs. We will follow Mackey's seminal works \cite{mackey1952induced, mackey1957borel} and the presentation in \cite[\S2.6]{folland1995course} (see also \cite[Chapter V, \S4]{varadarajan1968geometry}).

A {\bf group pair} \(C\lessdot G\) is a pair of an lcsc group \(G\) and a closed subgroup \(C\) of \(G\). Denote by \(G/C\) the coset space \(\left\{gC:g\in G\right\}\) with the quotient topology it inherits from \(G\). The group \(G\) acts continuously and transitively on \(G/C\) via
\[g.hC=ghC,\quad g\in G,\, hC\in G/C.\]
We will refer to measures on \(G/C\) simply as {\it invariant} or {\it quasi-invariant}, with the understanding that this always means invariance or quasi-invariance with respect to the natural action of \(G\) on \(G/C\). Since this action is transitive, it becomes ergodic with respect to whatever invariant or quasi-invariant measure we may put on \(G/C\).

\begin{thm}[Mackey]
For every group pair \(C\lessdot G\), the coset space \(G/C\) becomes a standard Borel space. Equivalently, with the natural action of \(C\) on \(G\), the orbit equivalent relation \(E_{C}^{G}\) is smooth.
\end{thm}

This theorem was proved by Mackey in two different formulations in \cite[Lemma 1.1]{mackey1952induced} and \cite[Theorem 7.2]{mackey1957borel}, which are equivalent due to Theorem~\ref{thm:smooth}. See also \cite[Corollary A.8]{zimmer2013ergodic}. In fact, it is true also beyond lcsc groups (see \cite[Theorem (12.17)]{kechris2012descriptive} and \cite[Theorem 2.2.10]{gao2008invariant}; cf. \cite[Proposition 3.4.6]{gao2008invariant}). A useful fact in proving the following extension of Mackey's theorem is that the quotient topology itself is Polish, making \(G/C\) a Polish \(G\)-space.

\begin{thm}
\label{thm:compstabsmooth}
Suppose \(C\lessdot G\) is a group pair with \(C\) compact, and let \(H\) be a closed subgroup of \(G\). Consider the natural action of \(H\) on \(G/C\). Then the orbit equivalence relation \(E_{H}^{G/C}\) is smooth.
\end{thm}

\begin{proof}
Since \(H\) is closed its action on \(G/C\) is continuous in the quotient topology, which is a Polish topology by \cite[Theorem 2.2.10]{gao2008invariant}. The orbits in this action are
\[H_{gC}=HgC,\quad gC\in G/C,\]
and, since \(H\) is closed and \(C\) is compact, we deduce that all orbits are closed. Finally, a Borel equivalence relation with closed equivalence classes is smooth (see \cite[Theorem (12.16)]{kechris2012descriptive}, \cite[Theorem 6.4.4(iv)]{gao2008invariant}, \cite[Proposition 2.1.12]{zimmer2013ergodic}).
\end{proof}

The construction of {\it measures on homogeneous spaces} yields a quasi-invariant Radon measure on \(G/C\). In the following we describe it as well as \emph{Weil's formula} it satisfies. Given a group pair \(C\lessdot G\), treating each of \(G\) and \(C\) as an lcsc group on its own right, we denote by \(\lambda_{G}\) and \(\lambda_{C}\) their respective left Haar measures and by \(\varDelta_{G}\) and \(\varDelta_{C}\) the corresponding modular functions. Every group pair \(C\lessdot G\) admits a {\bf rho-function}, i.e. a continuous function \(\rho:G\to\mathbb{R}_{>0}\) with
\[\rho\left(gc\right)=\frac{\varDelta_{C}\left(c\right)}{\varDelta_{G}\left(c\right)}\cdot\rho\left(g\right),\quad \left(g,c\right)\in G\times C.\]
(See \cite[Proposition 2.54]{folland1995course}). The following theorem is classical and in many occasions is formulated for locally compact groups that may not be second countable. We focus solely on lcsc groups.

\begin{thm}[Measures on homogeneous spaces]
\label{thm:meashom}
For a group pair \(C\lessdot G\):
\begin{enumerate}
    \item There is a unique, up to a positive constant multiple, invariant \(\sigma\)-finite Radon measure \(\kappa\) on \(G/C\) iff \(\varDelta_{G}\mid_{C}\equiv\varDelta_{C}\).
    \item There is always a unique, up to measure class, quasi-invariant \(\sigma\)-finite Radon measure \(\kappa\) on \(G/C\), with continuous Radon--Nikodym cocycle.
    \item \emph{Weil's formula}: For every \(\kappa\) as in (2) there is a rho-function \(\rho\) with
    \begin{equation}
    \begin{aligned}
    \label{eq:weil1}
    \int_{G}\varphi\left(g\right)d\lambda_{G}\left(g\right)=\int_{G/C}\Big[\int_{C}\varphi\left(gc\right){\textstyle \frac{\varDelta_{G}\left(c\right)}{\varDelta_{C}\left(c\right)}}d\lambda_{C}\left(c\right)\Big]\rho\left(g\right)^{-1}d\kappa\left(gC\right),
    \end{aligned}
    \end{equation}
    \[\text{ for every Borel function } \varphi:G\to\mathbb{R}_{\geq 0}.\]
\end{enumerate}
\end{thm}

\begin{proof}
Since \(G\) is assumed to be \(\sigma\)-compact, so is its continuous image \(G/C\). Thus, every Radon measure on \(G/C\) would be automatically \(\sigma\)-finite. Part (1) can be found in \cite[Theorem (2.49)]{folland1995course}. Part (2) can be found in \cite[Theorem (2.59)]{folland1995course}. As for part (3), in \cite[Theorem (2.56)]{folland1995course} there can be found the formula
\[\int_{G}\varphi\left(g\right)\rho\left(g\right)d\lambda_{G}\left(g\right)=\int_{G/C}\Big[\int_{C}\varphi\left(gc\right)d\lambda_{C}\left(c\right)\Big]d\kappa\left(gC\right)\]
for every Borel function \(\varphi:G\to\mathbb{R}_{\geq 0}\) (see \cite[\S2.7, pp. 64--65]{folland1995course}). Finally, for every Borel function \(\varphi:G\to\mathbb{R}_{\geq 0}\), substituting \(\varphi/\rho\) for \(\varphi\) in this formula and using the rho-function property, we obtain the formula in part (3). 
\end{proof}

The uniqueness of \(\kappa\) as stated in Theorem~\ref{thm:meashom}(2) is restricted to the class of Radon measures with continuous Radon--Nikodym cocycle. Also the famous uniqueness of the Haar measure is formulated in the common literature under further regularity property (e.g. Radon). We will need to strengthen this uniqueness.

It was shown by Mackey that quasi-invariance is sufficient to determine the measure class in the construction of measures on homogeneous spaces (a fact that is already significant when \(C\) is the trivial subgroup). Since every \(\sigma\)-finite measure is equivalent to a probability measure, which is always Radon, Mackey's theorem renders any additional regularity assumptions redundant. We formulate this for compact \(C\), where \(\varDelta_{G}\mid_{C}\equiv\varDelta\mid_{C}\) hence \(\kappa\) is invariant  according to Theorem~\ref{thm:meashom}(1).

\begin{thm}[Mackey]
\label{thm:haaruniq}
For a group pair \(C\lessdot G\) with \(C\) compact, with the invariant measure \(\kappa\) constructed in Theorem~\ref{thm:meashom}(1), the following hold:
\begin{enumerate}
    \item Every quasi-invariant \(\sigma\)-finite Borel measure on \(G/C\) is equivalent to \(\kappa\).
    \item Every invariant \(\sigma\)-finite Borel measure on on \(G/C\) is a positive constant multiple of \(\kappa\).
    \item The pushforward of \(\lambda_{G}\) along the canonical projection \(G\to G/C\) is a positive constant multiple of \(\kappa\).
\end{enumerate}
\end{thm}

\begin{proof}
Part (1). This is \cite[Theorem 1.1]{mackey1952induced} (cf. \cite[Theorem 7.1 and its Corollary]{mackey1957borel}).

Part (2). If \(\kappa'\) is any invariant \(\sigma\)-finite Borel measure equivalent to \(\kappa\), then \(d\kappa'/d\kappa\) is an invariant function on \(G/C\). Since the action is transitive and \(d\kappa'/d\kappa\) is an invariant function, it is necessarily a positive constant.

Part (3). Let \(q:G\to G/C\) be the canonical projection, which is continuous and open map, and consider the pushforward measure \(q_{\ast}\lambda_{G}\) on \(G/C\). This is a Borel measure and it is invariant since
\[q^{-1}\left(g.B\right)=gq^{-1}\left(B\right)\text{ for every Borel set }B\subseteq G/C\text{ and }g\in G.\]
Then with Part (2), in order to finish we have left to show that \(q_{\ast}\lambda_{G}\) is \(\sigma\)-finite (which fails when \(C\) is non-compact: \(G=\mathbb{R}^{2}\) and \(C=\mathbb{R}\)). For an arbitrary compact set \(K\subset G\) we have \(q^{-1}\left(K.C\right)\subseteq K\cdot C\) and then, since \(K\cdot C\) is compact,
\[q_{\ast}\lambda_{G}\left(K.C\right)\leq\lambda_{G}\left(K\cdot C\right)<+\infty.\]
As \(K\) is arbitrary and \(G\) is \(\sigma\)-compact, it follows that \(q_{\ast}\lambda_{G}\) is \(\sigma\)-finite.
\end{proof}

\subsection*{Acknowledgments}

We extend our heartfelt thanks to Michael Bj\"{o}rklund for several illuminating discussions and for his Remark~\ref{rmk:michael}. We also thank Sasha Danilenko for his advice, for answering our questions, and for pointing us to the reference~\cite{feldman1978orbit}. We are deeply grateful to Zemer Kosloff, who initiated this project, for his guidance, dedicated mentorship, and invaluable support. Finally, we thank Johanna Steinmeyer for her assistance in creating the figures.

\bibliographystyle{acm}
\bibliography{References}

\end{document}